\documentclass[11pt,letterpaper]{amsart}

\usepackage{cmap}  
\usepackage[T2A]{fontenc}
\usepackage[utf8]{inputenc}
\usepackage[english]{babel}
\usepackage[usenames]{color}
\usepackage{amsmath,amssymb,amsthm}
\usepackage{arcs}
\usepackage{epsfig}
\usepackage{filecontents}
\usepackage{graphicx}
\usepackage[pdftex,colorlinks=true,linkcolor=blue,urlcolor=red,unicode=true,hyperfootnotes=false,bookmarksnumbered]{hyperref}
\usepackage{ifthen}
\usepackage{import}
\usepackage{indentfirst}
\usepackage{mathtools}
\usepackage{cancel}
\usepackage{xcolor}
\usepackage{thmtools}
\usepackage{thm-restate}

\renewcommand{\le}{\leqslant}
\renewcommand{\ge}{\geqslant}

\newcommand{\ff}{\mathcal{F}}
\newcommand{\s}{\mathcal{S}}
\newcommand{\g}{\mathcal{G}}
\newcommand{\aaa}{\mathcal{A}}

\newcommand{\rr}{\mathcal R}
\newcommand{\m}{\mathcal}

\newcommand{\E}{\mathsf{E}}

\newcommand{\T}{\mathcal{T}}
\newcommand{\W}{\mathcal{W}}

\newtheorem{thm}{Theorem}

\newtheorem{lem}[thm]{Lemma}
\newtheorem{cor}[thm]{Corollary}
\newtheorem{cla}[thm]{Claim}
\newtheorem{prop}[thm]{Proposition}
\newtheorem{conj}{Conjecture}
\newtheorem{obs}[thm]{Observation}
\newtheorem{defn}[thm]{Definition}
\newtheorem{const}[thm]{Construction}

\title{The Hajnal--Rothschild problem}
\author{Peter Frankl}\address{R\'enyi Institute, Budapest, Hungary; Email: {\tt peter.frankl@gmail.com}}

\author{Andrey Kupavskii}
\address{Moscow Institute of Physics and Technology, Russia, St. Petersburg State University; Email: {\tt kupavskii@ya.ru}}

\date{}

\begin{document}
\maketitle
\begin{abstract}
For a family $\ff$ define $\nu(\ff,t)$ as the largest $s$ for which there exist $A_1,\ldots, A_{s}\in \ff$ such that for $i\ne j$ we have $|A_i\cap A_j|< t$. What is the largest family $\ff\subset{[n]\choose k}$ with $\nu(\ff,t)\le s$? This question goes back to a paper Hajnal and Rothschild from 1973. We show that, for some absolute $C$ and $n>2k+Ct^{4/5}s^{1/5}(k-t)\log_2^4n$, $n>2k+Cs(k-t)\log_2^4 n$ the largest family with $\nu(\ff,t)\le s$ has the following structure: there are sets $X_1,\ldots, X_s$ of sizes $t+2x_1,\ldots, t+2x_s$, such that for any $A\in \ff$ there is $i\in [s]$ such that $|A\cap X_i|\ge t+x_i$. That is, the extremal constructions are unions of the extremal constructions in the Complete $t$-Intersection Theorem. For the proof, we enhance the spread approximation technique of Zakharov and the second author. In particular, we introduce the idea of {\it iterative spread approximation}.
\end{abstract}
\section{Introduction}
Let $[n] = \{1,2,\ldots, n\}$ be the standard $n$-element set and $2^{[n]}$ its power set. For $k\le n$ let ${[n]\choose k}$ denote the collection of all the $k$-subsets of $[n]$. Collections $\ff\subset 2^{[n]}$ are called {\it families}. If $\ff\subset {[n]\choose k}$ then it is said to be $k$-uniform.

For a family $\ff\subset 2^{[n]}$ let $\nu(\ff)$ denote its {\it matching number}, the maximal number of pairwise disjoint members of $\ff$. If $\emptyset\in \ff$ then $\nu(\ff) = \infty$ but otherwise $\nu(\ff)\le n$. If $\nu(\ff) = 1$ then $\ff$ is called {\it intersecting}. In this case, one can finetune the definition and say that $\ff$ is {\it $t$-intersecting} if $|F\cap F'|\ge t$ for all $F,F' \in \ff$.

These are central notions for extremal set theory and many of its most important results/problems relate to these notions. For a detailed exposition we refer to the books \cite{GP} and \cite{FT}. Nevertheless, let us state some key results related to the present research.
\begin{thm}[Erd\H os, Ko and Rado \cite{EKR}]
  Let $k\ge t\ge 1$ integers and $n\ge n_0(k,t)$. Suppose that $\ff\subset {[n]\choose k}$ is $t$-intersecting. Then
  \begin{equation}\label{eq1.1}
    |\ff|\le {n-t\choose k-t}.
  \end{equation}
  Moreover, for $n>n_0(k,t)$ equality holds iff $\ff$ is a full $t$-star, i.e., for some $T\in {[n]\choose t}$, $\ff = \{F\in {[n]\choose k}: T\subset F\}$.
\end{thm}
Let us note that \eqref{eq1.1} was proved by Erd\H os, Ko and Rado \cite{EKR}, however, the exact value $n_0(k,t) = (k-t+1)(t+1)$ was determined by the first author for $t\ge 15$ \cite{F1} and by Wilson \cite{W} for $2\le t\le 14$, with a proof valid for all $t$. The first author conjectured that for all $n,k,t$ the largest $t$-intersecting family must be isomorphic to one of the following examples:
\begin{equation}\label{eqak}
  \mathcal D_i:=\Big\{F\in {[n]\choose k}: |F\cap [t+2i]|\ge i\Big\}.
\end{equation}
This was proved for all $n,k,t$ by Ahlswede and Khachatrian \cite{AK}. By now $t$-intersecting collections of objects were studied for various objects (cf. \cite{EFP, Kup54, Kup55, KuZa} and references therein).

The following problem is one of the most important unsolved questions in extremal set theory.
\begin{conj}[The Erd\H os Matching Conjecture, \cite{E}] Let $n,k,s$ be positive integers, $n\ge k(s+1)$. Let $\ff\subset {[n]\choose k}$ and suppose that $\nu(\ff)\le s$. then
\begin{equation}\label{eq1.2}
  |\ff|\le \max\Big\{{ks+k-1\choose k}, {n\choose k}-{n-s\choose k}\Big\}.
\end{equation}
\end{conj}
Let us note that the bounds in the RHS correspond to the construction ${[ks+k-1]\choose k}$ and $\{F\in {[n]\choose k}: F\cap [s]\ne \emptyset\}$. Since the first one is independent of $n$, for large $n$ the second one is dominant. Erd\H os \cite{E} proved
\begin{equation}\label{eq1.3}
  |\ff|\le {n\choose k}-{n-s\choose k} \ \ \ \text{for } n\ge n_1(k,s).
\end{equation}
The bounds on $n_1(k,s)$ were subsequently improved by Bollob\'as, Daykin and Erd\H os \cite{BDE}, Huang, Loh and Sudakov \cite{HLS}. The current best bounds are due to the present authors (cf. \cite{F4}, \cite{FK21}). We also note that, from the other end, the first author and then Kolupaev with the second author  showed that for $n$ close to $(s+1)k$ the first family is extremal (cf. \cite{F7}, \cite{KK}).

Let us also mention that the case $k=1$ is trivial. For $k=2$ ($k=3$) \eqref{eq1.2} was proved for all $n\ge k(s+1)$ by Erd\H os and Gallai \cite{EG} (by the first author in \cite{F6}), respectively.

Hajnal and Rothschild \cite{HR} proved a common generalization of \eqref{eq1.1} and \eqref{eq1.3}, which deals with a question that, rather surprisingly, was largely overlooked so far. In order to state their result, we need a definition and a construction.
\begin{defn} For a positive integer $t$ and a family $\ff\subset 2^{[n]}$, let $\nu(\ff,t)$ denote the maximum number $s$ such that there exist $F_1,\ldots, F_s\in \ff$ satisfying $|F_i\cap F_j|\le t-1$ for all $1\le i<j\le s$.
\end{defn}
Note that $\nu(\ff,1) = \nu(\ff)$ and $\nu(\ff,t)<\infty$ if $|F|>t$ for all $F\in \ff$. For a family $\ff$, sets $X\subset Y$ and a family  $\mathcal S$, let us introduce some notation, a part of which is somewhat non-standard.
\begin{align*}
\ff(X)&:=\{F\setminus X:X\subset F, F\in \ff\},\\
\ff[X]&:= \{F: X\subset F, F\in \ff\}, \\
\ff(X, Y)&:= \{F\setminus X: F\cap Y = X\},\\
\ff[X, Y]&:= \{F: F\cap Y = X\},\\
\ff(\s)&:=  \bigcup_{A\in \s}\ff[A],\\
\ff[\s]&:=  \bigcup_{A\in \s}\ff[A].
\end{align*}

\begin{const} Here and in what follows, let $\m A$ stand for the family ${[n]\choose k}$. Choose $\m M \subset {[n]\choose t}$, $|\m M| = s$, and define $\aaa[\m M]$ as above, that is,
$$\mathcal A[\m M] =\Big\{H\in {[n]\choose k}: M\subset H\text{ for some }M\in \m M\Big\}.$$
\end{const}
By the pigeon-hole principle, $\nu(\m A[\m M], t)\le s$ should be clear. For fixed $k,t,s$, and $n\to \infty$, we have $|\mathcal A[\m M]| = (1+o(1))s{n-t\choose k-t}$. Let us give a more general construction.
\begin{const}\label{const2} For each $i\in [s]$ choose a non-negative integer $x_i\le k-t$ and a set $Y_i$ of size $t+2x_i$. Define $\mathcal K:= \cup_{i\in[s]} {Y_i\choose t+x_i}$ and consider $\mathcal A[\m K]$. We call each ${Y_i\choose t+x_i}$ a $t$-intersecting clique and the family $\mathcal K$ a union of $s$ $t$-intersecting cliques.
\end{const}
Again, by the pigeon-hole principle, $\nu(\m A[\m K], t)\le s$ should be clear and, for fixed $k,t,s$, and $n\to \infty$, it is not difficult to check that we have $|\mathcal A[\m K]| = (1+o(1))(\sum_{i\in [s]}|\m D_{x_i}|)$ (cf. \eqref{eqak}) as long as $Y_{i_1}\nsubseteq Y_{i_2}$ for distinct $i_1,i_2$. At the same time, for $k>t\ge 2$ and $s\ge 1$ it is not obvious what choice $\m K$ maximizes $|\mathcal A[\m K]|$. It is unclear even for the simpler case of $\m M$.

Below we shall give a simple proof based on {\it shifting} showing that choosing $\m K$ so that $Y_i$ are pairwise disjoint  sets always attains the maximum. At the same time, even for the case of the matchings,  this is not the unique maximal choice. Namely, if  $2t-1>k$ and $|M_1\cup M_2|>k$ for all $M_1,M_2\in \m M$ then $|\mathcal H| = s{n-t\choose k-t}$. Equivalently, $|M_1\cap M_2|<2t-k$ for all $M_1,M_2\in \m M$. For the more general case of $\m K$ it seems to be a very challenging and tedious analytical problem to decide, which choice of sizes of cliques with disjoint supports maximizes $|\mathcal A(\m K)|$ for a each choice of $n,s,k,t$.\footnote{Most probably, it has no good answer. Even in the case of EMC we do not know, when exactly does the transition between two extremal examples happen.} Thus, we avoid trying to solve it here.  A reasonable guess is that, for any $n,s,k,t$, there is a non-negative integer $j$ such that only cliques with $x_i\in \{j,j+1\}$  appear.

For convenience, let $$h(n,k,t,s) = \max_{x_1,\ldots,x_s}|\m A[\m K]|,$$ where $\m K$ is as in Construction~\ref{const2} with pairwise disjoint $Y_i$. Also, for convenience, let $$h'(n,k,t,s) = |\m A[\m M]|,$$ where $\m M$ is an $s$-matching of $t$-element sets.

\begin{lem}\label{lem22}
  Fix positive integers $s,t,k,x_1,\ldots, x_s$ and assume $n\ge \sum_{i\in[s]}(t+2x_i)$. Let $\m K = \cup_{i\in[s]}{Y_i\choose t+x_i}$ for arbitrary $Y_1,\ldots, Y_s$ of sizes $t+2x_1,\ldots, t+2x_s$, respectively. Then $\m A[\m K]$ has size at most that of a family $\m A[\m K']$, which is a union of $s$ cliques with disjoint supports and same sizes as that of $Y_1,\ldots, Y_s$.
\end{lem}
\begin{proof}[Sketch of the proof] For any family $\m K$ the family $\m A[\m K]$ is actually the upper $k$-shadow of $\m K$. 
Let $\m K'$ be a family consisting of $s$ cliques on pairwise disjoint sets $Y'_1,\ldots, Y'_s\subset [n]$ of sizes $t+2x_1,\ldots, t+2x_s$, respectively. In order to prove the lemma, we note that, up to reordering of the elements, any family $\mathcal K$ can be obtained from  $\mathcal K'$ using a sequence of $S_{i\leftarrow j}$--shifts. Moreover, shifts do not increase the size of the upper shadow. 
\end{proof}

\begin{thm}[Hajnal and Rothschild \cite{HR}]\label{thmhr} Let $k>t\ge 1$, $s\ge 1$. Suppose that $\ff\subset {[n]\choose k}$ satisfies $\nu(\ff,t)\le s$. Then
\begin{equation}\label{eq1.4}
  |\ff|\le h'(n,k,t,s) \ \ \ \ \text{holds for }n>n_0(k,t,s).
\end{equation}
\end{thm}
The exact EKR theorem gives $n_0(k,t,1) = (k-t+1)(t+1)$. The current best results for the EMC \cite{F4,FK21} show $n_0(k,1,s)<2k(s+1)$ and $n_0(k,1,s)<\frac 53 sk$ for $s>s_0$. However, in great contrast, we do not know of any reasonably small bounds in the general case of the Hajnal-Rothschild Theorem. If one analyzes the proof in \cite{HR}, it requires that  $n_0(k,t,s)\ge {ks{k\choose t}\choose t+1}\approx k^{t^2}s^t$.

The main goal of the present paper is to extend the result of Hajnal and Rothschild with reasonable bounds on $n_0(k,t,s)$ and, moreover, to prove a more general Ahlswede--Khachatrian type result, determining the structure of the extremal example in the regime when families $\m K$ other than matchings maximize $\aaa[\m K]$. In what follows, $C$ is a non-specified absolute constant. 
\begin{thm}\label{thmmain}
Let $n,k,t\ge 2$ and $s\ge 1$ be integers such that $n>2k+ C(k-t)t^{4/5}s^{1/5}\log_2 n$, $n>2k+C(k-t)s\log_2^4 n$. If a family $\ff\subset {[n]\choose k}$ satisfies $\nu(\ff, t)\le s$ then  $|\ff|\le h(n,k,t,s)$ and, moreover, in case of equality $\ff=\m A[\m K]$ for a family $\m K$ that is a union of $s$ $t$-intersecting cliques as in Construction~\ref{const2}.
\end{thm}
The proof of the theorem relies on the spread approximation method introduced by Zakharov and the second author \cite{KuZa} and then developed in subsequent works by the second author \cite{Kup54,Kup55}. Here the method is greatly enhanced  with numerous extra ingredients. This result has several interesting features, which illustrate the strength of the method. First, the extremal example is not shifted. Without going into details, shifting is essential for many of the results on $t$-intersecting families and matchings, and in both cases the extremal examples are shifted. There is no known proof of the complete $t$-intersection theorem without shifting. Second, in many cases we could not even determine the exact form of the extremal example and thus do not have an expression for $h(n,k,s,t)$. Still, it is possible to prove Theorem~\ref{thmmain}.

What are the values of the parameters for which the  `trivial' example $\aaa[\m M]$ with $\m M$ being a matching becomes smaller than the example $\aaa[\m K]$, which is a union of $s$ vertex-disjoint cliques on $t+2$ vertices? It should be clear that, at least for moderately large $t$, it should be very close to $n = (k-t+1)(t+1),$ as in the case of the complete $t$-intersection theorem, since $|\aaa[\m M]|$ is very close to $s|\m D_0|$ and $|\aaa[\m K]|$ is very close to $s|\m D_1|$. The same is true for $\m A[\m K_i]$ vs $\m A[\m K_{i+1}]$, where $\m K_i$ is a union of $s$ cliques on disjoint supports, each on $t+2i$ vertices. But, again, the exact value of $n$ when the transition happens seems to be very challenging to determine. We only remark that the transition points seem to be essentially independent of $s$, as long as $s$ is not too large (i.e., $n>sk$ or so). Also, this analysis shows that Theorem~\ref{thmmain} indeed goes rather deep into the `non-trivial regime' in the case when $s$ is somewhat small compared to $t$ and $t$ itself is somewhat large. We also note that, on the other extreme, in the setup of the EMC, we only see two extremal families, and thus the assumption of $t$ being somewhat large w.r.t. $s$ is necessary in order to have families like $\m A[\m K_i]$ to be extremal.

\subsection{Organization of the paper} In Section~\ref{secspread} we define spread families and present some of their properties needed in the following sections. In Section~\ref{secsimple} we prove several simpler results concerning the Hajnal--Rothschild problem, and in particular an analogue of Theorem~\ref{thmhr} with polynomial bounds on $n$. The proof develops one of the approaches (the {\it peeling procedure}) that constitute the spread approximation method and works for a more general setting of quasirandom ambient families. In Section~\ref{secout} we present the outline of the proof of the main theorem. The following three sections contain the proof of the main theorem. Section~\ref{sec3} is devoted to obtaining a coarse approximation result ({\it iterative spread approximation} stage) for the Hajnal--Rothschild theorem, that outputs a family $\mathcal S$ of small uniformity with certain desirable quasirandom properties. Section~\ref{sec4} is devoted to discerning more fine-grained structure in families akin to $\mathcal S$. In Section~\ref{secmain} we put the results of Sections~\ref{sec3},~\ref{sec4} together and complete the proof of Theorem~\ref{thmmain}.

\section{Spread families}\label{secspread}
Given a real number $r>1$, we say that a family $\ff$ of sets is {\it $r$-spread} if for each set $X$ we have $|\ff(X)|\le r^{-|X|}|\ff|$. We say that $W$ is a {\it $p$-random subset} of $[n]$ if each element of $[n]$ is included in $W$ with probability $p$ and independently of others.

The spread approximation technique is based on the following theorem.
\begin{thm}[The spread lemma, \cite{Alw}, a sharpening due to \cite{Tao}]\label{thmtao}
  If for some $n,k,r\ge 1$ a family $\ff\subset {[n]\choose \le k}$ is $r$-spread and $W$ is a $(\beta\delta)$-random subset of $[n]$, then $$\Pr[\exists F\in \ff\ :\ F\subset W]\ge 1-\Big(\frac 5{\log_2(r\delta)} \Big)^\beta k.$$
\end{thm}
Stoeckl~\cite{Sto} showed that the constant $5$ can be replaced by $1 + h_2(\delta)$, where $h_2(\delta) = - \delta \log_2 \delta - (1 - \delta) \log_2 (1 - \delta)$ is the binary entropy of $\delta$. The lemma can be summarized as follows ``if a family is locally quasirandom (sufficiently spread), then it is globally quasirandom (we expect to see sets from the family inside a typical subset of the ground set).'' The parameters of interest here are so that $5/\log_2(r\delta) = 1/2$ and $\beta - \log_2 k \to \infty$, so that the probability of the described event becomes close to $1$. The choice of $\beta\delta$, and thus the condition on $r$, depends on the question, but, again, $r = C\log k$ already gives a meaningful conclusion.

An important observation is that any sufficiently large family contains a spread subfamily.
\begin{obs}\label{obs13}
  If $\ff\subset {[n]\choose \le k}$ and  $|\ff|>r^k$, then $\ff(X)$ is $r$-spread for some $X$, $0\le |X|<k$.
\end{obs}
\begin{proof}
  If $\ff$ is $r$-spread then we can put $X = \emptyset$. Otherwise, consider an inclusion-maximal $X$ such that $|\ff(X)|\ge r^{-|X|}|\ff|$. Note that $|X|\le k-1$ because for $|X|=k$ we have $|\ff(X)|=1<r^{-k}|\ff|$. By maximality of $X$, for any $Y$ that is disjoint with $X$ we have $|\ff(X\cup Y)|\le r^{-|X|-|Y|}|\ff|\le r^{-|Y|}|\ff(X)|.$ Thus, $\ff(X)$ is $r$-spread.
\end{proof}
The same proof implies the following
\begin{obs}\label{obs34}
  If $\ff\subset {[n]\choose \le k}$ and $X$ is an inclusion-maximal subset of the ground set so that $|\ff(X)|\ge r^{-|X|}|\ff|$, then  $\ff(X)$ is $r$-spread.
\end{obs}

\section{Simpler bounds}\label{secsimple}
In this section, we treat several particular cases of the HR problem. In particular, in the first subsection we show that the extremal example need not always have the form as in Construction~\ref{const2}. In the second subsection we treat the case $t=k-1$. In the third subsection, we extend the result of Hajnal and Rothschild to a much broader spectrum of parameters (with polynomial dependencies). The bounds on the parameters are more restrictive than in Theorem~\ref{thmmain}, but the proof is much easier and, moreover, it works for a broad spectrum of quasirandom families. (See the subsection for precise definitions.) It is as well based on some parts of the spread approximation method.
\subsection{$3$-uniform families}
\begin{prop}
  Suppose that $\m F\subset {[6]\choose 3}$. Then (i) and (ii) hold
  \begin{itemize}
    \item[(i)] If $\nu(\ff, 2)\le 2$ then $|\m F|\le 10$,
    \item[(ii)] If $\nu(\ff, 2)\le 3$ then $|\m F|\le 14$.
  \end{itemize}
\end{prop}
\begin{proof}
  If $\m F$ is intersecting then by the EKR theorem $|\m F|\le 10$. Assume next that $\m F$ is not intersecting, and w.l.o.g. $(1,2,3), (4,5,6)\in \ff$. For the remaining $18$ triples in $[6]$, let us partition them into $6$ groups of $3$ triples each:
  \begin{align*}
    \m P_1 &=\{(1,2,4),(2,3,5),(1,3,6)\}, \\
    \m P_2 &=\{(1,2,5),(2,3,6),(1,3,4)\}, \\
    \m P_3 &=\{(1,2,6),(2,3,4),(1,3,5)\}, \\
    \m R_i &=\{[6]\setminus T: T\in \m P_i\}, i\in[3].
  \end{align*}
  Note that $\m P_i\cup\{(4,5,6)\}\subset \ff$ would contradict $\nu(\ff,2)\le 3$. Hence, $|\ff\cap \m P_i|\le 2$ and similarly $|\ff\cap \m R_i|\le 2$. Consequently, $|\ff|\le 2+6\times 2 = 14$, proving (i).

  In the same way, in the case (i), $|\ff\cap \m P_i|\le 1$, $|\ff\cap \m R_i|\le 1$ and thereby $|\ff|\le 8$, i.e., $|\ff|<10$ if $\ff$ is not intersecting.
\end{proof}
Note that both the full star and ${Q\choose 3}$ with $|Q|=5$ show that (i) is best possible. In the case (ii) one can add all triples containing $(5,6)$ to ${[5]\choose 3}$ to provide an example that shows that $14$ is the maximum. Importantly, none of these families are of the form $\aaa[\m K]$ as in Construction~\ref{const2}.

\subsection{The $(k,k-1)$ case}
In this subsection, we study families $\ff\subset {[n]\choose k}$ with $\nu(\ff,k-1)\le s$.
If $n\ge 2s+k$ then we may fix $Q_1,\ldots Q_s$ such that $|Q_i| = k-1$ and $|Q_i\cap Q_j|\le k-3$ for each $i\ne j\in[s]$. Then the family $\aaa':=\aaa(\{Q_1,\ldots, Q_s\})$ has size $s(n-k+1)$ and satisfies $\nu(\aaa',k-1)\le s$.

\begin{thm}
  Let $n\ge 2k+2s-4+\max\{2s,k\}$. Assume that $\ff\subset {[n]\choose k}$ satisfies $\nu(\ff,k-1)\le s$. Then $|\ff|\le s(n-k+1)$.
\end{thm}
\begin{proof}
The proof is by induction on $s$. For $s=1$ this is simply an extremal statement concerning $k-1$-intersecting families.

  Take the largest family $\ff$ as in the theorem. Then we can fix $P_1,\ldots, P_s$ such that $|P_i\cap P_j|\le k-2$, $1\le i<j\le s$. Let $Y = P_1\cup \ldots\cup P_\ell$ and $v:=|Y|$. Set $Z = [n]\setminus Y$.

\begin{lem}\label{lemandr1}
 We may w.l.o.g. assume $|Y|\le k+2s-2$.
\end{lem}
\begin{proof}
  Consider a graph $X$ with $V(X) = \{P_1,\ldots, P_s\}$ and $E(X) = \{(P_i,P_j): |P_i\cap P_j| = k-2\}$.

  {\bf Case 1.} The graph $X$ is connected. Then take a spanning tree $T\subset X$ and w.l.o.g. assume that, for any $i$, $T\cap \{P_1,\ldots, P_s\}$ is connected. Then, it is easy to see by induction on $i$ that $|\cup_{j=1}^i P_j|\le k+2i-2$, and the lemma follows.

  {\bf Case 2.} The graph $X$ is not connected. W.l.o.g., let $i$ be such that for all $j_1,j_2\in [s]$ such that $j_1\le i<j_2$ we have $|P_{j_1}\cap P_{j_2}|\le k-3$. Then any set $F\in \ff$ cannot simultaneously satisfy $|F\cap G_{j_1}|\ge k-1$ and $|F\cap G_{j_2}|\ge k-1$ with some $j_1\le i<j_2$.

  Partition $\ff = \ff_1\sqcup \ff_2$, where
  $$\ff_1:=\big\{F\in \ff: \exists j\le i: |F\cap G_j|\ge k-1\big\},$$
  $$\ff_2:=\big\{F\in \ff: \exists j> i: |F\cap G_j|\ge k-1\big\}.$$
It is easy to see that $\nu(\ff_1,k-1)\le i$ and $\nu(\ff_2, k-1)\le s-i$. By induction, we may bound $|\ff_1|\le i(n-k+1)$ and $|\ff_2|\le (s-i)(n-k+1)$. Together, this implies the bound $|\ff|\le s(n-k+1)$.
\end{proof}

An {\it $(m,k-1)$-sunflower} is a collection of $m$ sets $A_1,\ldots, A_m$ of size $k$ s.t. $|\cap_{i=1}^m A_i| = k-1$.

\begin{lem}
  We may w.l.o.g. assume that $\ff$ has no $(2s+1,k-1)$-sunflower.
\end{lem}
\begin{proof}
  If $\m G$ is an inclusion-maximal $(m,k-1)$-sunflower $A_1,\ldots, A_m$, and $m\ge 2s+1$, then
  $$\nu(\ff\setminus \m G,k-1)\le s-1.$$
  Indeed, arguing indirectly, let $B_1,\ldots, B_s$ be a collection of sets in $\ff\setminus \m G$ with $|B_i\cap B_j|\le k-2$ for $i\ne j$. Denote $X:=\cap A_i$. Since $\m G$ is inclusion-maximal, we have $|B_i\cap X|\le k-2$ for $i\in [s]$. If for some $i$ we have $|B_i\cap X|\le k-3$, then $|B_i\cap A_j|\le k-2$ for all $j\in[m]$. If $|B_i\cap X|=k-2$ then there are at most two sets $A_{i_1}, A_{i,2}\in \m G$ such that $|B_i\cap A_{i_\ell}|=k-1$. Given that $m\ge 2s+1$, we can find $j\in [m]$ such that $|B_i\cap A_j|\le k-2$ for all $i\in[s]$. This contradicts the property $\nu(\ff,k-1)=s$.

  The displayed inequality by induction implies $|\ff\setminus \m G|\le (s-1)(n-k+1)$. Moreover, obviously $m\le n-k+1$, and we get the bound $|\ff|\le s(n-k+1)$.
\end{proof}

Put $Z:=[n]\setminus \cup_{i=1}^sP_i$. Note that, for each $z\in Z$ and $F\in\ff$, such that $z\in F$, there is $i$ such that $F\setminus \{z\}\subset G_i$.
Let us define $\ff_i:=\{G_i\}\cup \{F\in \ff: F\cap Z\ne \emptyset, F\setminus Z\subset G_i\}$. Then, $\ff_i$ is $k-1$-intersecting. Indeed, if there are $F, F'\in \ff_i$ with $|F\cap F'|\le k-2$, then  $F, F', P_j$, $j\in [s]\setminus\{i\}$ form a $k-1$-matching, a contradiction with $\nu(\ff,k-1)\le s$. It is thus easy to see that either $\ff_i\subset {G_i\cup \{z\}\choose k}$ for some $z\in Z$, or $\ff_i$ is an $(m,k-1)$-sunflower. But in the latter case, by the last lemma we know that $|\ff_i|\le 2s$. In any case, we have
$$|\ff_i|\le \max \{2s,k\}.$$
We conclude that
\begin{equation}\label{eqcompbound}|\ff|\le \Big|\ff\cap {Y\choose k}\Big|+s\max\{2s-1,k-1\}.\end{equation}
To bound the former term, we use the following lemma.
\begin{lem}\label{lempet4}
  We have $\big|\ff\cap {Y\choose k}\big| \le s|Y|$.
\end{lem}
  \begin{proof}
Put $r:=|Y|$ and partition all $k$-element sets in ${[r]\choose k}$ into the following $r$ groups:
    $$\m V_i:=\{(a_1,\ldots, a_k): a_1+\ldots+a_k = i \mod r\}.$$
    Note that any two sets from $\m V_i$ intersect in at most $k-2$ elements, and thus any family $\ff\subset {[r]\choose k}$ with $\nu(\ff,k-1)\le s$ satisfies $|\ff\cap \m V_i|\le s$ for each $i\in[r]$. Thus, $|\ff|\le rs$.
  \end{proof}

Substituting in \eqref{eqcompbound} the conclusions of Lemma~\ref{lempet4} and~\ref{lemandr1}, we get that $|\ff|\le s\big(k+2s-2+\max\{2s-1,k-1\}\big)$. Comparing it with the expression $s(n-k+1)$, we get that the latter is not smaller if $n\ge 2k+2s-4+\max\{2s,k\}$.
\end{proof}
\subsection{Peeling and a better bound in Theorem~\ref{thmhr}}
In this subsection, exceptionally, $\aaa$ stands not only for ${[n]\choose k}$, but for a class of `ambient' families. We say that a family $\s$ with $\nu(\s, t)= s$ is {\it non-trivial } if it is not of the form $\aaa[\m M]$ for a family $\m M$ that consists of $s$ sets of size $t$. We shall work with sufficiently quasirandom ambient families, and this quasirandomness is expressed in the notion of {\it $(r,f)$-spreadness}. Given a real number $r>1$ and a non-negative integer $f$, we say that a family $\aaa$ is {\it $(r,f)$-spread} if $\aaa(T)$ is $r$-spread for each $T$ such that $|T|\le f$. We say that a family $\T$ with $\nu(\T, t)=s$ is {\it maximal} if for any $T \in \T$ and any proper subset $X \subsetneq T$ the family $\T':=\T\setminus \{T\}\cup \{X\}$ we have $\nu(\T',t)>s$.

The following theorem is the main result of this subsection.
\begin{thm}\label{thmapproxtrivial}
Let $\varepsilon\in (0,1]$ and integers $n,r,q,m, t \ge 1$  be such that $\varepsilon r\ge 8esq$ and $q\ge t$.
Let $\aaa \subset 2^{[n]}$ be an $(r, t)$-spread family and let $\s \subset {[n] \choose \le q}$ be a maximal non-trivial family with $\nu(\s, t)=s$. Assume that $\s$ contains exactly $\ell$ sets $U_1,\ldots, U_\ell$ of size $t$.
Then there exists a $t$-element set $U$ such that
$|\aaa[\s]| \le \Big|\bigcup_{j=1}^\ell \aaa[U_j]\Big|+\varepsilon(m-\ell) |\aaa[U]|$.
\end{thm}
Before we go on to the proof of this theorem, let us deduce the corollary that would be necessary for the proof of the main theorem.
\begin{cor}\label{corapproxtrivial}
In terms of Theorem~\ref{thmapproxtrivial},  if $\aaa = {[n]\choose k}$ for $n\ge 24eskq$ and $\s\subset{[n]\choose \le q}$ is a non-trivial family with $\nu(\s,t)= s$ then $|\aaa[\s]|\le h'(n,k,t,s)-\frac 12 {n-t\choose k-t}$.

In particular, the conclusion of Theorem~\ref{thmhr} holds for any $n\ge 24esk^2$.
\end{cor}
\begin{proof}
  First, we note that, given sets $U_1,\ldots, U_s$, each of size $t$, and such that $U_s$ is disjoint with $U_1,\ldots, U_{s-1}$, we have
  \begin{multline}\label{mult1}|\aaa[U_s]\setminus (\cup_{i=1}^{s-1}\aaa[U_i])|\ge {n-t\choose k-t}-(s-1)t{n-t-1\choose k-t-1} \\
   = (1-\frac{(s-1)t(k-t)}{n-t}\big){n-t\choose k-t}\ge \frac 56{n-t\choose k-t},\end{multline}
  provided $n-t\ge 6(s-1)t(k-t)$. This is satisfied because of the imposed bound on $n$, since $q\ge t$.

  It is easy to check that $\aaa$ is $(\frac nk, t)$-spread. Assume that $\s$ is maximal and contains $\ell$ $t$-element sets $U_1,\ldots, U_\ell$. Since $\frac 13\cdot \frac nk \ge 8eskq$, we can use Theorem~\ref{thmapproxtrivial} with $\epsilon = \frac 13$ and $r=\frac nk$ and get that
  \begin{equation}\label{eq123}|\aaa[\s]| \le \Big|\bigcup_{j=1}^\ell \aaa[U_j]\Big|+\frac13(s-\ell){n-t\choose k-t}.\end{equation}
  Therefore, if we consider a family $\mathcal U:=\{U_1,\ldots, U_s\}$, where $U_1,\ldots, U_\ell$ are as in $\s$ and $U_{\ell+1},\ldots, U_s$ are pairwise disjoint and are disjoint from $U_1,\ldots, U_\ell$, by \eqref{mult1} and \eqref{eq123} we get that
  $$|\aaa[\mathcal U]| \ge \Big|\bigcup_{j=1}^\ell \aaa[U_j]\Big|+\frac 56(s-\ell){n-t\choose k-t}\ge |\aaa[\s]|+\frac12(s-\ell){n-t\choose k-t}.$$
  At the same time, we know that $h'(n,k,t,s)\ge |\aaa[\mathcal U]|$ by Lemma~\ref{lem22}.

  In order to obtain the second conclusion of the corollary, we take any family $\ff\subset {[n]\choose k}$ such that $\nu(\ff, t)\le s$ and then gradually replace sets by their subsets until we reach a maximal family $\s\subset {[n]\choose \le k}$ with $\nu(\s,t)\le s$. Clearly, $\aaa[\s]\supset \ff$. We than apply the first part of the corollary to $\s$.
\end{proof}
In what follows, we prove Theorem~\ref{thmapproxtrivial}. Given a family of sets $\s'$, let us say that $\T$ is a {\it simplification} of $\s'$ if $\T$ is obtained from $\s'$ by repeatedly applying the following steps: replace some set in the family by its proper subset and leave only inclusion-minimal sets in the resulting family.

Let us list some properties of simplifications.
\begin{enumerate}
\item $\T$ is an {\it antichain}, i.e., $A\not\subset B$ for any $A,B\in \T$;
\item each set $T\in \T$ is a subset of some set $S\in \s'$;
\item for each set $S\in \s'$ there is a set $T\in \T$ such that $T\subset S$;
\item for any family $\aaa$ we have $\aaa[\s']\subset \aaa[\T]$;
\item the size of the largest set in $\T$ is at most that in $\s'$;
\item for any family $\s'$ with $\nu(\s',t)\le s$ there exists a maximal family $\T$ with $\nu(\T,t)\le s$ that is a simplification of $\s'$.
\end{enumerate}
All these properties are immediate from the definition. 


Below, we describe a variant of the {\it peeling procedure}, introduced by Kupavskii and Zakharov in \cite{KuZa}. In the notation of Theorem~\ref{thmapproxtrivial}, let us iteratively define the following series of families.
\begin{enumerate}
    \item Put $\T_q=\s$. (Recall that $\s$ is a maximal family with $\nu(\s, t)=s$.)
    \item Put $\W_q = \T_q\cap {[n]\choose q}$, and note that $\nu(\T_q\setminus \W_q,t)\le s$. Moreover, $\T_q\setminus \W_q\subset {[n]\choose \le q-1}.$ Choose $\T_{q-1}\subset {[n]\choose \le q-1}$ to be  a maximal family with $\nu(\T_{q-1},t)\le s$ that is a simplification of $\T_q\setminus \W_q$.
    \item For $i = q-1,q-2,\ldots,t+1$ put $\W_i = \T_i \cap {[n] \choose i}$ and choose $\T_{i-1}\subset{[n]\choose \le i-1}$ to be a maximal family with $\nu(T_{i-1},t)\le s$ that is a simplification of  $\T_{i}\setminus \W_{i}\subset{[n]\choose \le i-1}$.
\end{enumerate}
Note that by defintion for each $i=q,\ldots,t$ each $\T_i$ is an antichain and we have $\nu(\T_i,t)\le s$. We summarize the properties of these series of families in the following lemma.

\begin{lem}\label{lemkeyred} Assume that for some $\ell\in \{0,\ldots, s-1\}$ $\T_q$ contains $\ell$ sets $U_1,\ldots, U_\ell$ of size $t$.  The following properties hold for each $i = q,q-1\ldots, t+1$. 
\begin{itemize}
 \item[(i)] All sets  in $\T_i$ have size at most $i$.
  \item[(ii)] There is no set $X$ of size  $<i$ 
      such that $\W_i(X)$ is $\alpha$-spread for $\alpha>s(i-t+1)$. 
  \item[(iii)] We have $|\W_i|\le (s-\ell)(2e s q)^{i-t}$.
 \item[(iv)] If $U_{\ell+1}\in \T_{i-1}\setminus \T_i$ is of 
     size $t$, then $|\aaa[\T_{i}[U_{\ell+1}]]|\le \frac{sq}r |\aaa[U_{\ell+1}]|$.

\end{itemize}
\end{lem}
\begin{proof}
(i) This immediately follows from the definition of $\T_i$.

(ii) Assume that there is such a set $X$. First, we consider the case $|X|=x\le t-1$. Then we are going to find sets $F_1,\ldots, F_{s+1}\in \W_i$ such that $F_i\supset X$ for each $i\in [s+1]$ and that violate the property $\nu(\W_i,t)\le s$. Take any $F_1'\in \W_i(X)$ and note that $|F_1'| = i-x$. Consider the families $\rr_1:=\{F\in \W_i(X): |F\cap F'_1|\ge t-x\}$ and $\g_1:=\W_i(X)\setminus \rr_1$. Using the fact that $\W_i(X)$ is $\alpha$-spread, for any given set $Y\subset F'_1$ of size $t-x$ we have $|\W_i(X\cup Y)|\le \alpha^{x-t}|\W_i(X)|$. Thus, $$|\rr_1|\le \alpha^{x-t}{i-x\choose t-x}|\W_i(X)|.$$ Next, for each $j=2,\ldots, s+1$ we select $F_j'\in \g_{j-1}$, put $\rr_{j}:=\{F\in \W_i(X): |F\cap F'_j|\ge t-x\}$ and $\g_j:=\g_{j-1}\setminus \rr_j$. Note that the displayed upper holds for $|\rr_j|$ as well. Provided that $\g_1,\ldots, \g_s$ are non-empty, we can select $F_1',\ldots, F_{s+1}' \in \W_i(X)$ such that $|F'_j\cap F'_{j'}|<t-x$ for each $j\ne j'$. Then put $F_j:=X\cup F_j'$ and note that $F_1,\ldots, F_{s+1}\in \W_i$ and that these sets violate $\nu(\W_i,t)\le s$. We are only left to verify that the families $\g_i$, $i=1,\ldots, s$, are non-empty. For this, it is actually sufficient to check that $\sum_{i=1}^s |\rr_i|<|\W_i(X)|$. Using the displayed bound, it is implied by $s\alpha^{x-t}{i-x\choose t-x}<1$. But the latter is true, since
$$s{i-x\choose t-x}= s\prod_{j=1}^{t-x}\frac{i-t+j}{j}\le s (i-t+1)^{t-x}<\alpha^{t-x}.$$

 Next, consider the case $|X|\ge t$. By maximality of $\T_i$, there must be sets $F_1,\ldots, F_s\in \T_i$ such that $X,F_1,\ldots, F_s$ violate $\nu(\T_i,t)\le s$. That is, we have $|X\cap F_j|<t$ for each $j\in[s]$ and thus $|F_j\setminus X| = |F_j|-|F_j\cap X|\ge |F_j|-t+1$.
For each $j\in [m]$, choose $F'_j\subset F_j\setminus X$, such that $|F'_j| = |F_j|-t+1 \le i-t+1$.  Then any $G\in \W'_i:=\W_i[X, X\cup F_1'\cup\ldots\cup F_s']$ intersects each $F_i$, $i\in [s]$, in at most $t-1$ elements. Indeed, $F_i\cap G\subset F_i\setminus F_i'$, and the latter set has size $t-1$. Provided that $\W'_i$ is non-empty and we can choose such a $G$, we get a contradiction with the fact that  $\nu(\T_i,t)\le s$. (Indeed, recall that, by our choice of $F_j$, $|F_{\gamma}\cap F_{\gamma'}|<t$ for any distinct $\gamma,\gamma'\in [s]$.)  We have $|F_1'\cup\ldots\cup F_s'| \le s(i-t+1)<\alpha$, and thus, by the $\alpha$-spreadness of $\W_i(X)$, we have \begin{align*}|\W'_i|\ge&\ |\W_i(X)|-\sum_{x\in \cup_{i=1}^s F'_i} |\W_i(X\cup \{x\})|\\
\ge&\ |\W_i(X)|- \frac{s(i-t+1)}\alpha |\W_i(X)|>0.\end{align*}

(iii) Let $t+1\le i\le q$. Recall that the family $\T_i$ is an antichain because it is a simplification. Thus, for each $W\in \W_i$ and $j\in [\ell]$ we have $U_j\not \subset W$.
Let us show that there is a  set $Y\in \W_i$ that $t$-intersects at least a $\frac 1{s-\ell}$-fraction of sets in $\W_i$. We argue indirectly. If each set $Y\in \W_i$ $t$-intersects strictly smaller than a $\frac 1{s-\ell}$-fraction of sets in $\W_i$, then we can greedily select sets $Y_1,\ldots, Y_{s-\ell+1}\in \W_i$ so that $|Y_j\cap Y_{j'}|<t$ for each $j,j'\in[s-\ell+1]$, $j\ne j'$. But then for the sets $U_1,\ldots, U_\ell, Y_1,\ldots, Y_{s-\ell+1}$ all pairwise intersections have size strictly smaller than $t$, which contradicts the fact that $\nu(\T_i,t)\le s$. In what follows, we fix $Y\in \W_i$ that $t$-intersects at least a $\frac 1{s-\ell}$-fraction of sets in $\W_i$.

 We have $|Y|=i$. By averaging, there is a $t$-element subset $Z\subset Y$ such that $|\W_i| \le (s-\ell){i \choose t} |\W_i(Z)|=(s-\ell){i \choose i-t} |\W_i(Z)|$. Next, $\W_i(Z)$ is $(i-t)$-unform and, moreover, by Part (ii) of the lemma, there is no $X, X\cup Z=\emptyset$, such that $\W_i(X\cup Z)$ is $\alpha$-spread with $\alpha>s(i-t+1)$. By (the contrapositive of) Observation~\ref{obs13} we get that $|\W_i(Z)|\le (s(i-t+1))^{i-t}$.

From here we conclude that
\begin{align*}
|\W_i| \le& (s-\ell){i \choose i-t} |\W_i(X)|\\
\le& (s-\ell)\Big(\frac{ei}{i-t}\Big)^{i-t}(s(i-t+1))^{i-t}\\
\le& (s-\ell)\Big(\frac{ei}{i-t}\Big)^{i-t}(2s(i-t))^{i-t}\\
\le& (s-\ell)(2e s i)^{i-t}\le  (s-\ell)(2e s q)^{i-t}.
\end{align*}

(iv)  Note that all sets in $\T_{i}$ containing $U_{\ell+1}$ have size at least $t+1$ by definition. Also recall that, by definition, $\T_i$ is an antichain. 
Recall that, for a family $\ff$, $\tau(\ff)$ is the size of the smallest set $Y$ such that $Y\cap F\ne \emptyset$ for each $F\in \ff.$ Arguing indirectly, assume that $\tau(\T_i(U_{\ell+1}))> qs$. By maximality of $\T_i$, there are sets $F_1,\ldots, F_s\in \T_i$ such that $U_{\ell+1},F_1,\ldots, F_s$ have pairwise intersections strictly smaller than $t$. At the same time, $|F_1\cup\ldots \cup F_s|\le qs$, and thus there is a set $F\in \T_i[U_{\ell+1}]$ such that $F\cap F_i = U_{\ell+1}\cap F_i$ for each $i\in[s]$. But then $F, F_1,\ldots, F_s\in \T_i$ and  have pairwise intersections strictly smaller than $t$,  which contradicts $\nu(\T_i,t)\le s$. We conclude that $\tau(\T_i(U_{\ell+1}))\le qs$.
If $\{x_1, \ldots, x_{sq}\}$ intersects all sets in $\T_i(U_{\ell+1})$ then we have
$$
|\aaa[\T_i[U_{\ell+1}]]| \le |\aaa[U_{\ell+1}\cup \{x_1\}]| + \ldots + |\aaa[U_{\ell+1}\cup \{x_{sq}\}]| \le \frac {sq}r |\aaa[U_{\ell+1}]|.
$$
The second inequality is due to $(r,t)$-spreadness of $\aaa$.
\end{proof}

\begin{proof}[Proof of Theorem~\ref{thmapproxtrivial}] 
Fix $\phi\ge t$ to be the largest index such that  $\T_{\phi}$ consists only of sets of size $t$. Note that by Lemma \ref{lemkeyred}, (i) such a choice of $\phi$ always exists. Also note that $\T_{\phi}$ consists of $s'\le s$ sets of size $t$ because $\nu(\T_{\phi},t)\le s$. Put $\T_{\phi} = \{U_1\ldots, U_{s'}\}$ and for each $j\in [s']$ define $f(j)$ to be the largest index $i'$ such that $U_{j}\in \T_{i'}$. Recall that $\s=\T_q$ contains exactly $\ell$ sets  $U_1,\ldots, U_\ell$ of size $t$. Next, crucially, we have
\begin{equation}\label{eqdecomp}\aaa[\s]= \aaa[\T_q]\subset \bigcup_{j=1}^\ell \aaa[U_j] \cup\bigcup_{i=\phi+1}^q\aaa[\W_i] \cup \bigcup_{j=\ell+1}^{s'} \aaa[\T_{f(j)+1}[U_j]].\end{equation}
The first equality is trivial because $\s=\T_q$. 
Let us verify the validity of the inclusion. Take any set $F\in \T_q$ and let us show that there is $F'\subset F$ such that $$F'\in \{U_j: j\in[\ell]\}\cup \bigcup_{i=\phi+1}^q\W_i\cup \bigcup_{j=\ell+1}^{s'} \T_{f(j)+1}[U_j].$$ Note that this automatically implies \eqref{eqdecomp}.

If $U_j\subset F$ for $j\in [\ell]$ or if there is an $i=\phi+1,\ldots, q$ and a set $W\in\W_i$ such that $W\subset F$, then we are done. Assume that neither holds for $F\in \T_q$. Then, by property (3) of simplifications and the definition of $\T_i$, there exists a (not necessarily unique) chain of sets $F=F_0\supset F_1\supset\ldots\supset F_N= U_j$, where: $F_i\in \T_i$ for each $i$; we have $j>\ell$;  we have $1\le N=f(j)\le \phi$.  It should then be clear that, first, $F_{N-1}\subset F$ and, second, $F_{N-1}\in \T_{f(j)+1}[U_j]$. This shows validity of the claim and of \eqref{eqdecomp}.



From the $(r,t)$-spreadness of $\aaa$, we get $$|\aaa[\W_i]|\le |\W_i||\aaa[W]|\le |\W_i|r^{-(i-t)}|\aaa[U']|\le |\W_i|r^{-(i-t)}|\aaa[U]|,$$ where $W\in \W_i$ is such that $|\aaa[W]|$ is the largest,  $U'\subset W$ is a  $t$-element subset of $W$ and $U$ is a $t$-element set such that $|\aaa[U]|$ is the largest. Combining this and the inclusion \eqref{eqdecomp} with Lemma~\ref{lemkeyred}, we get
\begin{align*}
|\aaa[\s]| \le\ & \Big|\bigcup_{j=1}^\ell \aaa[U_j]\Big|+\sum_{j=\ell+1}^{s'}\big| \aaa[\T_{f(j)+1}[U_j]] \big| +\sum_{j=\phi+1}^{q}\big|\aaa[\W_j]\big|\\
\overset{(iii),(iv)}{\le}& \Big|\bigcup_{j=1}^\ell \aaa[U_j]\Big|+ \Big(\frac {sq (s'-\ell)} r +\sum_{j=1}^{\infty} r^{-j}(s-\ell)(2e sq)^{j} \Big)|\aaa[U]|\\ \le  \ &\Big|\bigcup_{j=1}^\ell \aaa[U_j]\Big|+(s-\ell)\Big(\frac \varepsilon 4 +\sum_{j=1}^{\infty} \big(\frac \varepsilon4\big)^j \Big)|\aaa[U]|\\
\le\ & \Big|\bigcup_{j=1}^\ell \aaa[U_j]\Big|+ (s-\ell)\varepsilon |\aaa[U]|,
\end{align*}
where the third inequality uses $\varepsilon r \ge 8esq$.
\end{proof}

\section{Proof outline }\label{secout}  Fix a family $\ff$ with $\nu(\ff,t)\le s$. The proof of Theorem~\ref{thmmain} consists of three stages. In the first stage, we use the technique of spread approximations in order to find a family $\m S$ of sets of small size such that $\nu(\m S,t)\le s$ and $\ff\setminus\ff[\m S]$ is small. The key point here is that the uniformity of sets in $\m S$ is just slightly bigger than $t$. This result is given in Theorem~\ref{thmapprox1} in Section~\ref{sec3}. The proof takes up  some ideas from \cite{KuZa} and enhances them with several new ideas. The main one is an efficient procedure to find a dense piece inside the family and then apply the spread approximation procedure once more. It is combined with a relaxation of $\nu(\m S,t)\le s$ to $\nu(\m S,t')\le s$ with a smaller $t'$, which allows to work in a much broader parameter range. The argument uses an iterative bootstrapping which interchangeably applies these two steps: we improve the density of the piece that we find (Theorem~\ref{thm2sets}), and then, because of the improved density property, reduce the uniformity of $\m S$ and improve $t'$ (Theorem~\ref{thmregularity} and Lemma~\ref{lemtint}), then going back to the density step and so on.

In the second stage, we show that $\m S$ has the structure as in Construction~\ref{const2}: it is the union of $s$ $t$-intersecting cliques. The corresponding statement is Theorem~\ref{thmfinegrained} in Section~\ref{sec4}. The proof is by induction on $s$: basically, we find one such clique, show that the remainder has $t$-matching number of size $\le s-1$ and apply induction. In order to find the $t$-intersecting clique, we 
find the first `dense' largest-uniformity layer in $\m S$  (the sets in $\m S$ have uniformity between $t$ and $t+\ell$ with $\ell\ll t$). Then, we discern an approximate `union of $s$ $t$-intersecting cliques' structure, and use size bounds in order to show that one of these cliques is actually rather dense. The latter allows to show that the remainder of $\m S$ should have $t$-matching number at most $s-1$.

In the third stage, we combine some of the ideas from the previous sections and the ideas in \cite{Kup55} to prove that the remainder $\m R$ obtained in the spread approximation is empty. A key step is to actually find a set $X$ such that $\m R(X)$ is sufficiently spread. This allows us to reduce the uniformity with which we work, but also relates the size of $\m R(X)$ and $\m R$.

\section{Spread approximation for Hajnal--Rotschild families}\label{sec3}

 The goal of the following `spread approximation' theorem is to approximate any family $\ff$ with $\nu(\ff,t)\le s$ by a family $\m S$ with $\nu(\m S,t)\le s$  of much lower uniformity (cf. Subsection~\ref{secout}). The proof of this result occupies the rest of this section.  

\begin{restatable}{thm}{approxone}
\label{thmapprox1}
  Let $n,k,s,t$ be integers, $\sigma\ge 0$ a real number, and assume that
  $n\ge t+C(st)^{1/2}(k-t)\log_2 \frac{n-t}{k-t}$,
  $n\ge k+C(k-t)s(\sigma+\log_2 n)$ for some absolute constant $C$.
  Let $\ff\subset {[n]\choose k}$ satisfy $\nu(\ff,t)\le s$.
  Then there exists a family $\m S$ of sets of size at most
   $t+(t/s)^{1/2}+3\sigma+\log_2(16 s^4t^2)$ 
  such that $\nu(\m S,t)\le s$ and, moreover, the family $\m R:= \ff\setminus \ff[\m S]$ satisfies $|\m R|\le 2^{-\sigma} \cdot 15\log_2(st+\sigma)\cdot {n-t\choose k-t}$.

  More generally, if additionally $n\ge t+Ct^{\alpha}s^{\beta}(k-t)\log_2 \frac{n-t}{k-t}$ with some $\alpha,\beta\ge 1/2$, then the bound on the sizes of sets in $\m S$ improves to at most $t+t^{(1-\alpha)}s^{-\beta}+3\sigma+\log_2(16 s^4t^2).$
\end{restatable}


One can see that the sets in $\m S$ have size just slightly above $t$, which is extremely helpful in the next, peeling step. E.g., if $n\ge C(k-t)\max\{s,t\}\log_2 n$, then the sets in $\m S$ have size at most $t+400\log_2\frac{n-t}{k-t}$.

\subsection{A substructure in the family $\m G$ with $\nu(\m G,t)\le s$}\label{sec22}
The goal of this subsection is to prove the following theorem.
\begin{thm}\label{thm2sets} Let $n,k,s,t_1,\ell$ be integers and let $\lambda>0$ be a real number. Assume $\ell\ge t_1$.  Let $\m G$ be a family of $\le \ell$-element sets that satisfies $\nu(\m G,t_1)\le s$. Let $\ff[\m G]\subset \m A[\m G]$ satisfy $|\ff[\m G]|> \lambda{n-t_1\choose k-t_1}$. Then there exists a set $X$ of size $x$,
$$t_1\le x\le t_1+ 4\Big(\frac{(k-t_1)t_1(\ell-t_1)^2}{(n-t_1)}\Big)^{1/3}+ \log_2\Big(s^2t_1\lambda^{-1}\Big)$$
such that, denoting $|\m F(X)| = \beta {n-x\choose k-x}$,
\begin{equation}\label{eqsize2set}
  \m F[\m G]\le 8s^2t_1e^{3\big(\frac{(k-t_1)t_1(\ell-t_1)^2}{n-t_1}\big)^{1/3}} \beta{n-t_1\choose k-t_1}.
\end{equation}
\end{thm}
Roughly speaking, if the family $\ff[\m G]$ is large then we can relate the size of $\ff[\m G]$ and $\ff[X]$, and thus conclude that the latter is rather dense (much denser than $\ff[\m G]$ itself).

Consider a family $\m G$ of $\le \ell$-sets such that $\nu(\m G,t_1)\le s$. Let us first prove two lemmas that allow us to find a small substructure for the subsequent analysis. Note that in this subsection we will make use of the graph terminology, with the sets from the family playing the role of vertices.
For a set $A\in \m G$ let the degree $d_{\m G}(A)$ stand for the number of sets from $\m G$ that intersect $A$ in at least $t_1$ elements, and let $e(\m G)$ be the number of pairs of sets from $\m G$ that intersect in at least $t_1$ elements.

\begin{prop}\label{lem10}
  If $\m G'$ satisfies $\nu(\m G',t_1)\le s$ then 
  the number of edges is at least $s{|\m G'|/s\choose 2}$. 
\end{prop}
\begin{proof}
  It is just an application of Turan's theorem: what is the minimum number of edges in a graph with independence number $s$ and $|\m G'|$ vertices? It is at least the number of edges into a disjoint union of $s$ cliques of nearly-equal size, which is at least $s{|\m G'|/s\choose 2}$, where we postulate that ${x\choose 2} = 0$ for $x\le 1$. 
\end{proof}

\begin{lem}\label{lem11} Consider a family $\m G'$ such that $\nu(\m G',t_1)\le s$ for some positive integers $t_1,s$.
\begin{itemize}
\item[(i)] The number of triples (triangles) $A,B,C$ such that $|A\cap B|,|B\cap C|, |A\cap C|\ge t_1$ is at least $$\frac s{3}\sum_{A\in \m G'}{d_{\m G'}(A)/s\choose 2}.$$
\item[(ii)] 
There is a pair of sets $A,B$ in $\m G'$ such that $|A\cap B|\ge t_1$ and such that there are at least
    $$\frac {|\m G'|}{4s^2}$$
    sets $C\in \m G'$ such that $|C\cap A|, |C\cap B|\ge |A\cap B|$.
    \end{itemize}
\end{lem}
\begin{proof} (i) Take a set $A\in \m G'$ and consider the family $\m G_A$ of its neighbors (that is, sets that intersect $A$ in at least $t_1$ elements). In the family $\m G_A$, there are at least $s{|\m G_A|/s\choose 2}$ 
edges by Proposition~\ref{lem10}. Each such edge gives us a triangle with vertex $A$. Summing it up over all vertices and taking into account that each triangle is counted $3$ times, we get the result.

(ii) Note that ${x\choose 2}$ (with the convention that ${x\choose 2} = 0$ for $x\le 1$) is convex. Applying Jensen's inequality to the expression in the first part of the lemma, we get that there are
\begin{align*}\frac s{3}\sum_{A\in \m G'}{d_{\m G'}(A)/s\choose 2}&\ge \frac {s|\m G'|}{3}{\sum_{A\in \m G'} d_{\m G'}(A)/(s|\m G'|)\choose 2}\\
&=\frac {s|\m G'|}{3}{2e(|\m G'|)/(s|\m G'|)\choose 2} = \frac{e(|\m G'|)}3\Big(\frac{2e(|\m G'|)}{s|\m G'|}-1\Big)\end{align*}
triangles in $\m G'$. Each triangle consists of sets $A,B,C$ out of which (at least) one of the pairs has the smallest intersection. Assuming that $|\m G'|\ge 8s^2$, there exists an edge $A,B$ of sets from $\m G'$ such that there are at least
$$\frac{2e(|\m G'|)}{3s|\m G'|}-\frac 13\ge \frac{2s{|\m G'|/s\choose 2}}{3s|\m G'|}-\frac 13=\frac 13\Big(\frac{|\m G'|}{s^2}-\frac{s+1}s\Big)\ge \frac{|\m G'|}{4s^2}$$
triangles $A,B,C$ with $C\in \m G'$ and such that $|C\cap A|, |C\cap B|\ge |A\cap B|$. Here in the first inequality we used the bound on the number of edges in $\m G'$ from Proposition~\ref{lem10} and in the last inequality the bound $|\m G'|\ge 8s^2$.

But if $|\m G'|<8s^2$, it is enough to find $(8s^2)/(4s^2) = 2$ such sets, and both $A$ and $B$ satisfy the condition.
\end{proof}

Apply Lemma~\ref{lem11} (ii) to $\m G$ playing the role of $\m G'$ with $t_1$.  We get two sets $A,B\in \m G$ as in the lemma. Assume that they intersect in $t'\ge t_1$ elements.  Choose $I\subset A\cap B$, such that $|I| = t_1$, and put $D_1:=A\setminus B$, $D_2:=B\setminus A$. Note that $|D_1|,|D_2|\le \ell-t_1$. Then for each of the at least $|\m G|/(4s^2)$ sets $C$ as guaranteed by Lemma~\ref{lem11} there is a value $i\in \{0,\ldots, t_1\}$ and sets $U,V,W$ such that $|U| = t_1-i$, $|V|=|W| = i$ and $C\cap I = U, C\cap D_1\supset V, C\cap D_2\supset W$.   By the bound on number of such $C$ from Lemma~\ref{lem11} (ii) and the pigeon-hole principle, there is a choice of $i$ and such sets $U,V,W$ such that
$$\frac{|\m G(U\cup V\cup W)|}{|\m G|}\ge \frac 1{4 s^2(t_1+1){t_1\choose t_1-i}{\ell-t_1\choose i}^2}.$$
Indeed, for any given $i$ there are  ${t_1\choose t_1-i}{\ell-t_1\choose i}^2$ ways to choose (parts of) the intersections $U,V,W$ of $C$ with $I,D_1,D_2$, and thus for one of the $t_1+1$ possible values of $i$ and for one of the corresponding choices $U,V,W$ we must get the inequality above.  Let us put $X = U\cup V\cup W$  and define $\beta$ so that $|\ff(X)| = \beta {n-|X|\choose k-|X|}$. Remark that $|X|=t_1+i$. Next, we bound the value of $i$ using the bound on the size of $\m F[\m G]$. We have
\begin{align}\notag |\m F[\m G]|\le& 4 s^2(t_1+1){t_1\choose i}{\ell-t_1\choose i}^2 \beta {n-t_1-i\choose k-t_1-i}\\
\notag \le& 8 s^2t_1{t_1\choose i}{\ell-t_1\choose i}^2 \Big(\frac {k-t_1}{n-t_1}\Big)^{i}\beta {n-t_1\choose k-t_1}\\
\label{eq65} \le& 8 s^2t_1 \Big(\frac{e^3(k-t_1)t_1(\ell-t_1)^2}{(n-t_1)i^3}\Big)^i \beta{n-t_1\choose k-t_1}
\end{align}
In the last inequality we use the inequality ${x\choose m}\le (ex/m)^m$, valid for any $x\ge m\ge 1$. If $i\ge 4\Big(\frac{(k-t_1)t_1(\ell-t_1)^2}{(n-t_1)}\Big)^{1/3}+ \log_2\Big(s^2t_1\lambda^{-1}\Big)$, then the RHS of \eqref{eq65} is at most
$$\beta\lambda{n-t_1\choose k-t_1},$$
which contradicts our assumption on the size of $|\m F[\m G]|$. Thus, we may assume that the opposite inequality on $i$ holds. In particular, this implies the inequality on $|X|$ from the theorem. The maximum of \eqref{eq65} is attained for $i = \Big(\frac{(k-t_1)t_1(\ell-t_1)^2}{n-t_1}\Big)^{1/3},$ which gives the bound
$$|\m F[\m G]|\le 8s^2t_1e^{3\big(\frac{(k-t_1)t_1(\ell-t_1)^2}{n-t_1}\big)^{1/3}} \beta{n-t_1\choose k-t_1}.$$
This completes the proof of the theorem.

\subsection{Spread approximation}
The next theorem is a variant of the spread approximation theorem which makes use of dense pieces within our family. The idea is that, rather than search for a spread approximation for $\ff$ itself, we find a set $X$ on which $\ff(X)$ is dense and then a spread piece inside it. Then we remove the latter piece from $\ff$ and repeat. The gain is in the bound on the remainder $\m R$, stated in part (iii).
\begin{thm}\label{thmregularity}
  Let $\eta,\theta,\ell_1,\ell_2, r>0$. Assume that for a family $\ff\subset {[n]\choose k}$ and any its subfamily $\m P$ of size at least $\eta$ there is a set $X$ of size $\ell_1\le x\le \ell_2$ such that $|\m P| \le \theta\frac{|\m P(X)|}{{n-x\choose k-x}}$. Let $q\ge \ell_2$ and $r\le \frac{n-\ell_1}{k-\ell_1}$. Then there is a family $\m S$ of sets of size at most $q$ and a family $\m R\subset \ff$ such that the following holds.
  \begin{itemize}
    \item[(i)] $\ff= \m R\sqcup \bigsqcup_{S\in \m S} \m F_S[S]$;
    \item[(ii)] for any $A\in \s$ and the family $\ff_A\subset \ff$ the family $\ff_A(A)$ is $r$-spread;
    \item[(iii)] $|\m R|\le \max\big\{\eta, \theta \cdot \frac{r^{q+1-\ell_2}{n-q-1\choose k-q-1}}{{n-\ell_2\choose k-\ell_2}}\big\}$.
  \end{itemize}
\end{thm}
\begin{proof}
Consider the following procedure for $i=1,\ldots $ with $\ff^1:=\ff$.
\begin{enumerate}
    \item If $|\ff^i|<\eta$ then stop.
    \item Take a set $X_i$ for $\ff^i$ as guaranteed in the statement of the theorem.
    \item Find a maximal $S_i\supset X_i$ that  $|\ff^i(S_i)|\ge  r^{|X_i|-|S_i|}|\ff^i(X_i)|$.
    \item If $|S_i|> q$  then stop. Otherwise, put $\ff^{i+1}:=\ff^i\setminus \ff^i[S_i]$.
\end{enumerate}
Note that Observation~\ref{obs34} and maximality of $S_i$ imply that $\ff^i(S_i)$ is $r$-spread. 
Let $N$ be the step of the procedure for $\ff$ at which we stop. The family $\s$ is defined as follows: $\s:=\{S_1,\ldots, S_{N-1}\}$. Clearly, $|S_i|\le q$ for each $i\in [N-1]$. The family $\ff_{A}$ promised in (ii) is defined to be $\ff^i[S_i]$ for $A=S_i$. Next, we put $\m R:=\ff^N$. Note that if $|\ff^N|\ge \eta$, then 
\begin{align*}|\ff^N|\le&\ \theta\frac{|\ff^N(X_N)|}{{n-|X_N|\choose k-|X_N|}}\le \theta\frac{ r^{|S_N|-|X_N|}  |\ff^{N}(S_N)|}{{n-|X_N|\choose k-|X_N|}}\\
\le&  \theta \Big(\frac{r(k-\ell_1)}{n-\ell_1}\Big)^{\ell_2-|X_N|}\cdot\frac{ r^{|S_N|-\ell_2}  |\ff^{N}(S_N)|}{{n-\ell_2\choose k-\ell_2}}\\
\le&  \theta \frac{ r^{|S_N|-\ell_2}  |\ff^{N}(S_N)|}{{n-\ell_2\choose k-\ell_2}}\le  \theta \frac{ r^{|S_N|-\ell_2}  {n-|S_N|\choose k-|S_N|}}{{n-\ell_2\choose k-\ell_2}}\\
\le& \theta \frac{ r^{q+1-\ell_2}  {n-q-1\choose k-q-1}}{{n-\ell_2\choose k-\ell_2}}\end{align*}
Since either $|S_N|>q$ or $|\ff^N|<\eta$, we get the inequality on $|\m R|$.
\end{proof}

\begin{lem}\label{lemtint} Let $\ff\subset {[n]\choose k}$ satisfy $\nu(\ff,t)\le s$. Next, let $\mathcal S\subset {[n]\choose \le \ell}$, $\ell\ge t$, have a property that for each $A\in \mathcal S$ there is a subfamily $\ff_A\subset \ff$ such that $\ff_A(A)$ is $r$-spread. Assume that for some $t'$, $t' \le t$ the following conditions are satisfied: \begin{align}\label{eqint1}r^{\big\lceil\frac{t-t'+1}2\big\rceil}\ge& 2s \Big(\frac{2e(\ell-t'+1)}{t-t'+1}\Big)^{\big\lceil\frac{t-t'+1}2\big\rceil},\\
\label{eqint2} r\ge& 2^{12}(s+1)\log_2(4k).
\end{align}
Then $\nu(\m S, t')\le s$. \end{lem}
\begin{proof}

  Take $s+1$ (not necessarily distinct) $A_1,\ldots, A_{s+1}\in \mathcal S$  and assume that $a_{ij}:=|A_i\cap A_j|<t'$ for all $i\ne j\in [s+1]$. 
  Put $p = \frac{t-t'}2$ and for $i\in [s+1]$ consider
  $$\g_i:=\ff_{A_i}(A_i)\setminus\Big\{F\in \ff_{A_i}(A_i): \ \exists j\in[s+1], \text{ such that } |F\cap  (A_j\setminus A_i)|\ge \frac{t-a_{ij}}2\Big\}.$$ The size of the latter family is at most
 \begin{align*}\sum_{j\in [s+1]\setminus\{i\}}{| A_j\setminus A_i|\choose \big\lceil\frac{t-a_{ij}}2\big\rceil}\max_{X: |X| = \big\lceil\frac{t-a_{ij}}2\big\rceil, X\cap A_i = \emptyset}|\ff_{A_i}(A_i\cup X)|\le&\\  \sum_{j\in [s+1]\setminus\{i\}}{\ell-a_{ij}\choose \big\lceil\frac{t-a_{ij}}2\big\rceil} r^{-\big\lceil\frac{t-a_{ij}}2\big\rceil} |\ff_{A_i}(A_i)|\le&\\
  \sum_{j\in [s+1]\setminus\{i\}} \Big(\frac{e(\ell-a_{ij})} {\big\lceil\frac{t-a_{ij}}2\big\rceil} \Big)^{\big\lceil\frac{t-a_{ij}}2\big\rceil} r^{-\big\lceil\frac{t-a_{ij}}2\big\rceil}|\ff_{A_i}(A_i)|\le&\\
  \sum_{j\in [s+1]\setminus\{i\}} \Big(\frac{2e(\ell-t'+1)}{t-t'+1} \Big)^{\big\lceil\frac{t-a_{ij}}2\big\rceil} r^{-\big\lceil\frac{t-a_{ij}}2\big\rceil}|\ff_{A_i}(A_i)|
  \le&\\
  s\Big(\frac{2e(\ell-t'+1)}{r(t-t'+1)} \Big)^{\big\lceil\frac{t-t'+1}2\big\rceil} |\ff_{A_i}(A_i)|\le&\ \frac 12 |\ff_{A_i}(A_i)|,\end{align*}
  where we pass from the third to the fourth line using $\frac{e(\ell-a_{ij})} {\lceil\frac{t-a_{ij}}2\rceil}\le \frac{2e(\ell-a_{ij})} {t-a_{ij}}\le \frac{2e(\ell-t'+1)} {t-t'+1}$  since $\ell\ge t$ and $a_{ij}\le t'-1$. In the last inequality we use  \eqref{eqint1}.

  This implies that $|\g_i|\ge \frac 12 |\ff_{A_i}(A_i)|$.
 Because of this and the trivial inclusion $\g_i(Y)\subset \ff_{A_i}(A_i\cup Y)$, valid for any $Y$,  we conclude that $\g_i$ is $ \frac r2$-spread for all $i\in [s+1]$. By \eqref{eqint2}, we have $\frac r2 > 2^{11}(s+1)\log_2(4k)$.  

  What follows is an application of Theorem~\ref{thmtao}. Let us put $\beta= \log_2(4k)$ and $\delta = (2(s+1)\log_2(4k))^{-1}$. Note that $\beta\delta = \frac 1{2(s+1)}$ and $\frac r2\delta > 2^{10}$ by \eqref{eqint2}.  Theorem~\ref{thmtao} implies that a $\frac{1}{2(s+1)}$-random subset $W_i$ of $[n]\setminus A_i$ contains a set from $\g_i$ with probability strictly bigger than
  $$1-\Big(\frac 5{\log_2 2^{10}}\Big)^{\log_2 (4k)} k = 1-2^{-\log_2 (4k)} k = \frac 34.$$
Next, note that by monotonicity the event we described is contained in (actually, coincides with) the following event:
\begin{quote}
  A $\frac{1}{2(s+1)}$--random subset $W_i$ of $[n]$ contains a set $G_i$ for some $G_i\in \g_i$.
\end{quote}
Thus the probability of the latter event is also at least $\frac 34$. Consider the upwards closed families $$\m Q_i:=\{F\subset [n]: \exists G\in \g_i \text{ such that } G \subset F\},$$
$\m Q_i\subset 2^{[n]}$. Then the last lower bound on the probability is equivalent to saying $\mu_{1/2(s+1)}(\m Q_i)\ge \frac 34$ for each $i\in [s+1]$. We shall apply a variant of the result from Huang Loh and Sudakov \cite{HLS} due to Keevash, Lifshitz, Long and Minzer \cite{KLLM}.
We say that a family $\m G$ is \textit{upward closed} if for any $A \in \m G$ and $B, A \subset B,$ one has $B \in \m G$.

\begin{lem}
\label{lemhls}
Fix $p_1,\ldots, p_s$ so that $\sum_{i = 1}^{s+1} p_i \le 1/2$. Let $\m Q_1, \ldots, \m Q_{s+1}$ be upward closed subfamilies of $2^{[n]}$ such that
    \begin{align*}
        \mu_{p_i} (\m Q_i) \ge 3 (s+1) p_i.
    \end{align*}
    Then there exist disjoint sets $Q_1 \in \m Q_1, \ldots, Q_{s+1} \in \m Q_{s+1}$.
\end{lem}
Applying this Lemma to $\m Q_1,\ldots, \m Q_{s+1}$ with $p_i = 1/2(s+1)$ implies that there exist a `rainbow' matching $Q_i\in \m Q_i$, $i\in [s+1]$. Therefore, there are sets $G'_i\in \m G_i$, $i\in [s+1]$, such that $G'_i\cap G'_j = \emptyset $. Consider the $(s+1)$-tuple $G_1,\ldots, G_{s+1}$, where $$G_i = G_i'\cup A_i.$$ We know for $G_i,G_j$ that their parts on $[n]\setminus (A_i\cup A_j)$ are  disjoint. On the other hand, by the definition of $\g_i$, we have $|G_{i}\cap A_j\setminus A_i|, |G_{j}\cap A_i\setminus A_j|<\frac {t-a_{ij}}2$, and thus
\begin{align*}|G_{i}\cap G_{j}|\le& |A_{i}\cap A_{j}|+|G_{i}\cap (A_j\setminus A_i)|+ |G_{j}\cap (A_i\setminus A_j)|\\
<& a_{ij}+\frac {t-a_{ij}}2+\frac {t-a_{ij}}2=t\end{align*}
 for distinct $i,j$. We have $G_1,\ldots, G_{s+1}\in\ff$, which contradicts the fact that $\nu(\ff,t)\le s$.
\end{proof}

For convenience, let us restate the main theorem of this section. 
\approxone

\subsection{Proof of Theorem~\ref{thmapprox1}} Let us first deal with the first part of the theorem, i.e., the case $\alpha=\beta=0.5$.
  If $k\le t+(t/s)^{1/2}+2\sigma+\log_2(16 s^4t^2)$ then the family $\ff$ itself can serve as the family $\m S$. In what follows, we assume that $k> t+(t/s)^{1/2}+2\sigma+\log_2(16 s^4t^2)$.
  The proof of the theorem is a bootstrapping argument that goes back and forth between Theorem~\ref{thm2sets} and a combination of Theorem~\ref{thmregularity} with Lemma~\ref{lemtint}.

  If $k\ge 10t+t\log_2 (st)+\sigma$ then we first perform steps A(0), B(0). If not, then we skip these and put $\m S^{(0)} = \ff$, $q^{(0)} = 10t+t\log_2 (st)+\sigma$ and $t'^{(0)} = t-\lceil t^{1/2}\rceil$.

   {\bf Step A(0) } We apply Theorem~\ref{thmregularity} to $\ff$  with $\ell_1 = \ell_2 = x=0$, $\theta = 1$, $\eta = 0 $, $r = 2^{9}(st)^{1/2}\big(1+\frac \sigma t\big)\log_2(st)$, $q=q^{(0)} = t+t\log_2(2^{9}(st))+\sigma = 10t+t\log_2(st)+\sigma$.
    (Basically, we ignore the part of the statement that looks
    for the `dense' part of $\ff$ and do a `normal' spread
    approximation.) We find a family $\m S^{(0)} = \m S$ of
    $(\le 10t+t\log_2(st)+\sigma)$-uniform sets and a family $\m R$ such that

\begin{align*}|\m R|&\le \Big(2^{9} (st)^{1/2}\big(1+\frac \sigma t\big)\log_2(st)\Big)^{q}{n-q\choose k-q}\\&\le
\Big(2^{9} (st)^{1/2}\big(1+\frac \sigma t\big)\log_2(st)\Big)^{q} \Big(\frac{k-t}{n-t}\Big)^{q-t}{n-t\choose k-t}\\
&\le \Big(2^{9} (st)^{1/2}\big(1+\frac \sigma t\big)
\log_2(st)\Big)^{q} \Big(\frac{1}{2^{12} (st)^{1/2}\big(1+\frac \sigma t\big)\log_2(st)}\Big)^{q-t}{n-t\choose k-t}\\
& =
8^{-q+t} \Big(2^9 (st)^{1/2}\big(1+\frac \sigma t\big)\log_2(st)\Big)^{t}{n-t\choose k-t}\\
&\le
2^{-\sigma}{n-t\choose k-t}.\end{align*}
   Here the third inequality is due to the bound $n\ge t+C(st)^{1/2}\big(1+\frac\sigma t\big)(k-t)\log_2\frac{n-t}{k-t}$, which in particular implies $\log_2\frac{n-t}{k-t}\ge \frac 12\log_2(st)$. The
   last inequality is due to our choice of $q$ and the fact that $(1+\frac \sigma t)^t\le 4^\sigma$.

   {\bf Step B(0) } For each $A\in \m S$ there is a family $\ff_A\subset \ff$ such that $\ff_A(A)$ is $r$-spread. Let us apply Lemma~\ref{lemtint} to the spread decomposition of $\ff$ constructed in Step A(0) with $t'=t'^{(0)} = t-\lceil t^{1/2}\rceil $, $\ell=q^{(0)}$ and the same $r$. It is easy to see that \eqref{eqint2} is satisfied. As for \eqref{eqint1}, it is sufficient to have
  $$\Big(2^9(st)^{1/2}\big(1+\frac \sigma t\big)\log_2(st)\Big)^{\lceil\frac{t^{1/2}+1}2\rceil}\ge 2s\Big(\frac{2e\ell}{\lceil t^{1/2}\rceil +1}\Big)^{\lceil\frac{t^{1/2}+1}2\rceil},$$
  implied by our choice of $\ell$
  and the inequality $\lceil\frac{t^{1/2}+1}2\rceil\ge 2$, which is valid in our assumptions. (Note that the last inequality is equivalent to $t> 1$.) 
  We conclude that $\m S$ satisfies $$\nu\big(\m S, t-\lceil t^{1/2}\rceil \big)\le s.$$

  The next series of bootstrapping steps alternate Theorem~\ref{thm2sets}, in which we improve upon the bound on $|\m F|$, and a combination of Theorem~\ref{thmregularity} with Lemma~\ref{lemtint}, in which we improve upon the uniformity of $\m S$ and the value of $t'$.
   For each $i = 1,\ldots, i_0$ (where $i_0$ shall be defined below) perform the following steps.
For the moment, let us assume  that $k-t$ is not too small in the following sense:
         \begin{equation}\label{eqsmallk} k-t'^{(0)}\le 2(k-t).\end{equation}
   That is, $k\ge t+\lceil t^{1/2}\rceil $.
  \begin{itemize}
    \item[{\bf Step A(i)}] Apply Theorem~\ref{thm2sets} to each subfamily $\ff'$ of $\ff$ with size bigger than $2^{-\sigma}{n-t\choose k-t}$ with $\m S^{(i-1)}$ playing the role of $\m G$; $t_1^{(i)} = t'^{(i-1)}$ playing the role of $t_1$;  $\ell^{(i)} = q^{(i-1)}$ playing the role of $\ell$;, $\lambda^{(i)} = \frac{2^{-\sigma} {n-t\choose k-t}}{\binom{n-t_1^{(i)}}{k-t_1^{(i)}}}$ playing the role of $\lambda$.

       We will maintain $t_1^{(1)}\le t_1^{(2)}\le\ldots\le t_1^{(i_0)}\le t$. We assumed \eqref{eqsmallk}, and, given that $t_1^{(i)}\ge t_1^{(1)}=t'^{(0)}$, we have $k-t_1^{(i)}\le 2(k-t)$. Using the bound \begin{equation}\label{eqnst} (n-t_1^{(i)})/((k-t_1^{(i)})t_1^{(i)})\ge (n-t)/(2(k-t)t) \ge 2^{10}(s/t)^{1/2}\log_2\frac{n-t}{k-t}\end{equation} in the second line below, for each $\ff'$ we get a set $X$ so that the size $x$ of $X$ satisfies
       {\small \begin{align*} t_1^{(i)}\le x\le& t_1^{(i)}
       +4\Big(\frac{(k-t^{(i)}_1)t_1^{(i)} (q^{(i-1)}-t_1^{(i)})^{2}}{n}\Big)^{1/3}+ \log_2\big(s^2t(\lambda^{(i)})^{-1}\big)\\
        \le& t+(t/s)^{1/6}\Big(\frac{q^{(i-1)}-t_1^{(i)}}{100\log_2\frac{n-t}{k-t}}\Big)^{2/3} +\sigma+\log_2(s^2t)+(t-t_1^{(i)})\log_2 \frac{n-t}{k-t}:=\ell_2^{(i)},
        \end{align*}}
        where for the last summand we used that $\lambda^{(i)}\ge 2^{-\sigma}\big(\frac{k-t}{n-t}\big)^{t-t_1^{(i)}}$. We  also get the relation between the sizes of the families $\ff'[\m S^{(i-1)}]$ and $\ff'[X]$. Putting  $|\ff'[X]| = \beta^{(i)}{n-x\choose k-x},$ we have
            \begin{align*}
          |\m F[\m S^{(i-1)}]|&\le 8s^2t_1^{(i)}e^{3\big(\frac{(k-t_1^{(i)})t_1^{(i)} (q^{(i-1)}-t_1^{(i)})^2}{n-t_1^{(i)}}\big)^{1/3}} \beta^{(i)}{n-t_1^{(i)}\choose k-t_1^{(i)}}\\
          &\overset{\eqref{eqnst}}{\le} 8s^2t2^{(t/s)^{1/6}\big(\frac{q^{(i-1)}-t_1^{(i)}}{100\log_2\frac{n-t}{k-t}}\big)^{2/3}} \beta^{(i)}{n-t_1^{(i)}\choose k-t_1^{(i)}}
        \end{align*}
    \item[{\bf Step B(i)}] Apply Theorem~\ref{thmregularity} to $\ff$ with $\ell_1^{(i)} = t_1^{(i)}$ playing the role of $\ell_1$; $\ell_2^{(i)}$ as defined in Step A(i) above playing the role of $\ell_2$; $$\eta^{(i)}=2^{-\sigma} {n-t\choose k-t} = \lambda^{(i)}\binom{n-t_1^{(i)}}{k-t_1^{(i)}}$$
        playing the role of $\eta$;
        $$\theta^{(i)} = 8s^2t2^{(t/s)^{1/6}\big(\frac{q^{(i-1)}-t_1^{(i)}}{100\log_2\frac{n-t}{k-t}}\big)^{2/3}} {n-t_1^{(i)}\choose k-t_1^{(i)}}$$
        playing the role of $\theta$;
        $r=\frac{n-t_1^{(i)}}{2(k-t_1^{(i)})}$, and
        {\small \begin{align}\notag q^{(i)} =& \ell_2^{(i)}+(t/s)^{1/6}\Big(\frac{q^{(i-1)}-t_1^{(i)}}{100\log_2\frac{n-t}{k-t}}\Big)^{2/3}+ \log_2\frac{8s^2t {n-t_1^{(i)}\choose k-t_1^{(i)}}}{2^{-\sigma}{n-t\choose k-t}}\\
        \label{boundq}\le& t_1^{(i)}+2(t/s)^{1/6}\Big(\frac{q^{(i-1)}-t_1^{(i)}} {100\log_2\frac{n-t}{k-t}}\Big)^{2/3} +2\sigma+\log_2(8s^4t^2)+2(t-t_1^{(i)})\log_2 \frac{n-t}{k-t}.
        \end{align}}
         Note that, by Step A(i), for any family $\ff'[S^{(i-1)}]$ (playing the role of $\m P$) of size at least $\eta$ the condition on the size of $\ff'[S^{(i-1)}]$ vs the size of $\ff'[X]$ is automatically satisfied by our choice of $\theta^{(i)}$.  The condition on $r$ is also clearly satisfied, and thus Theorem~\ref{thmregularity} is indeed applicable. Below we show that the value of $q^{(i)}$ is chosen in such a way that $$|\m R^{(i)}|\le 2^{-\sigma}{n-t\choose k-t}.$$ Recall Theorem~\ref{thmregularity} (iii). The first expression in the maximum, $\eta^{(i)}$, immediately gives exactly this upper bound. The second expression in the maximum in the RHS of Theorem~\ref{thmregularity} (iii) is at most $$2^{-(q^{(i)}+1-\ell_2^{(i)})}\theta^{(i)}=2^{-\sigma} {n-t\choose k-t}.$$

         We also get a spread decomposition of $\ff$ via a family $\m S^{(i)}$ of uniformity at most $q^{(i)}$. We now can apply Lemma~\ref{lemtint} to $\ff$ and $\m S^{(i)}$ with $r=\frac{n-t_1^{(i)}}{2(k-t_1^{(i)})}$ and any
         $t'^{(i)},$ $t'^{(i-1)}\le t'^{(i)}\le t$ that plays the role of $t$, provided that $t'^{(i)}$ satisfies \begin{equation}\label{eqchoicet} \frac{q^{(i)}-t'^{(i)}+1}{t-t'^{(i)}+1} \le \frac{C' \big(\sigma+(ts)^{1/2}\log_2\frac{n-t}{k-t}\big)}{s^{1/\lceil (t-t'^{(i)}+1)/2\rceil}}\end{equation} for some large constant $C'<C$.  It is easy to that \eqref{eqint2} is satisfied since $n\ge C(k-t)s\log_2 n$ and thus $r\ge \frac C4 s\log_2 n$ by \eqref{eqsmallk}. As for condition  \eqref{eqint1}, it is satisfied whenever the inequality \eqref{eqchoicet} above is satisfied. Indeed, remark that, again in the assumption \eqref{eqsmallk}, we have $n-t\ge \frac C4 (st)^{1/2}(k-t_1^{(i)})\log_2 \frac {n-t}{k-t}+\frac C4 s(k-t_1^{(i)})(\sigma+\log_2 \frac {n-t}{k-t})$, and so the value of $r$ we chose is at least the numerator in the RHS of \eqref{eqchoicet}.

         As long as $t'^{(i)}\le t-2\log_2 s$, say, the RHS of \eqref{eqchoicet} is at least
         \begin{equation}\label{eqchoicet2}\frac {C'}2 \big(\sigma+(ts)^{1/2}\log_2\frac{n-t}{k-t}\big),\end{equation}
         and thus \eqref{eqchoicet} is implied by $t-t'^{(i)}+1\ge \frac{q^{(i)}-t'^{(i)}+1}{\frac {C'}2 \big(\sigma+(ts)^{1/2}\log_2\frac{n-t}{k-t}\big)}$. We put

           \begin{equation}\label{boundt} t'^{(i)} = \min\Bigg\{t-\lceil 2\log_2 s\rceil , t-\Big\lfloor \frac{q^{(i)}-t_1^{(i)}+1}{\frac {C'}2 \big(\sigma+(ts)^{1/2}\log_2\frac{n-t}{k-t}\big)}\Big\rfloor\Bigg\}.\end{equation}
 \end{itemize}
Note that the second expression in the minimum automatically satisfies inequality \eqref{eqchoicet} provided $t'^{(i)}=t^{(i+1)}_1\ge t^{(i)}_1$ (note that we have $t_1^{(i)}$ instead of $t'^{(i)}$, also cf. Step A(i)). In what follows, we analyze the dynamics of the   values $t-t'^{(i)}$ and $q^{(i)}-t'^{(i)}$ as the process moves forward up until $t'^{(i)}$ reaches the value $t-\lceil 2\log_2s\rceil $. It is sufficient to analyze the inequality in \eqref{boundq} and the equality in \eqref{boundt}. It should be clear that $t'^{(i)}$ does not decrease as a function of $i$ provided the value of $q^{(i)}$ does not increase as $i$ grows.

Let us show that
$q^{(i)}-t^{(i)}_1\le \frac 23 (q^{(i-1)}-t^{(i)}_1)$ provided
\begin{equation}\label{qdec} q^{(i-1)}-t^{(i)}_1\ge (t/s)^{1/2}+4\sigma+2\log_2(8s^4t^2)+ 8(1+\log_2 s)\log_2\frac{n-t}{k-t}.\end{equation}
First, remark that the displayed inequality in
particular implies $(t/s)^{1/6}\le (q^{(i-1)}-t^{(i)}_1)^{1/3}$.
Second, remark that, by \eqref{boundt} we have \begin{align*}t-t^{(i)}_1 = t-t'^{(i-1)}&=\max\Big\{\lceil 2\log_2 s\rceil, \Big\lfloor \frac{q^{(i-1)}-t_1^{(i-1)}+1}{\frac {C'}2 \big(\sigma+(ts)^{1/2}\log_2\frac{n-t}{k-t}\big)}\Big\rfloor\Big\}\\
&\le \lceil 2\log_2 s\rceil+\Big\lfloor \frac{q^{(i-1)}-t_1^{(i-1)}+1}{\frac {C'}2 \big(\sigma+(ts)^{1/2}\log_2\frac{n-t}{k-t}\big)}\Big\rfloor\\
&\le \lceil 2\log_2 s\rceil+1+\Big\lfloor \frac{q^{(i-1)}-t+1}{\frac {C'}2 \big(\sigma+(ts)^{1/2}\log_2\frac{n-t}{k-t}\big)}\Big\rfloor\\
&\le 2\log_2 s+2+0.01 \frac{q^{(i-1)}-t_1^{(i)}}{\log_2\frac{n-t}{k-t}},\end{align*}
where the last inequality holds provided $C'$ is large enough.
Substituting these two facts into \eqref{boundq} for $i$, we get
{\small $$q^{(i)}-t^{(i)}_1\le \frac 1{10}(q^{(i-1)}-t^{(i)}_1)+0.02 (q^{(i-1)}-t_1^{(i)}) +2\sigma+\log_2(8s^4t^2)+(4+4\log_2 s)\log_2\frac{n-t}{k-t}.$$}
Note that, by \eqref{qdec}, $2\sigma +\log_2(8s^4t^2)+(4+4\log_2 s)\log_2\frac{n-t}{k-t} \le \frac 12(q^{(i-1)}-t^{(i)}_1)$. Therefore, the displayed equation above implies $q^{(i)}-t^{(i)}_1\le \frac 1{10}(q^{(i-1)}-t^{(i)}_1)+0.02 (q^{(i-1)}-t_1^{(i)}) +\frac 12 (q^{(i-1)}-t_1^{(i)})$, and the RHS is smaller than $\frac 23(q^{(i-1)}-t_1^{(i)})$.

Thus, as long as \eqref{qdec} is valid, the value of $q^{(i)}$ decreases, and we have $q^{(i)}-t^{(i+1)}\le q^{(i)}-t^{(i)}\le \frac 23 (q^{(i-1)}-t^{(i)})$, so the difference between $q^{(i)}-t^{(i+1)}$ decays exponentially with $i$. In particular, we should reach the point when \eqref{qdec} is not valid any more in at most $10\log_2 (st+\sigma)$ steps, since $q^{(0)}-t^{(0)}\le 10t+\log_2(st)+\sigma$ and $(10t+\log_2(st)+\sigma)(2/3)^{10\log_2 (st+\sigma)}<1$. At this point, say, $i_0$, it is easy to see that the second expression in the minimum in the RHS of \eqref{boundt} is simply $t$ (since the fraction in the lower integer part is smaller than $1$), and thus $t'^{(i_0)} = t-\lceil 2\log_2 s\rceil$.

Then, we run Steps A(i) and B(i) again, with $t'^{(i)} = t-2$ for $i_0+1\le i\le i_0+\lceil\log_2 s\rceil=:i_1$ to get rid of the extra $\log_2 s$ factor. In order to satisfy \eqref{eqint1}, we need
$$n-t\ge C'(k-t)s^{1/2} \Big((t/s)^{1/2}+4\sigma+2\log_2(8s^4t^2)+8(1+\log_2 s)\log_2\frac{n-t}{k-t}\Big),$$ which is satisfied in our assumption on $n$. Repeating these steps up until step $i_1$ and analyzing how the value of $q^{(i)}-t_1^{(i)}$ evolves using \eqref{boundq} as before, we see that the value $q^{(i_1)}$ is at most
$ t+(t/s)^{1/2}+4\sigma+2\log_2(8s^4t^2)+ 8\log_2\frac{n-t}{k-t}$. (That is, the `extra $\log_2 s$ evaporated'.) Then we apply the second step with $t'^{(i_1)} = t$ and see that in order to satisfy \eqref{eqint1}, we need
$$n-t\ge C'(k-t)s \Big((t/s)^{1/2}+4\sigma+\log_2(8s^4t^2)+\log_2\frac{n-t}{k-t}\Big),$$
which is also valid in our assumptions. Plugging this into the A-step, we get a bound $q^{(i_1+1)}\le  t+(t/s)^{1/2}+3\sigma+\log_2(16 s^4t^2)$.

Thus, $\mathcal S^{(i_1)}$ gives us the desired approximation. Let us bound the size of the remainder. At every application of Theorem~\ref{thmregularity}, we had $|\m R^{(i)}|\le 2^{-\sigma}{n-t\choose k-t}$, and the number of steps was clearly at most $15\log_2(st+\sigma)$, and thus, over all, steps we accumulated at most $2^{-\sigma} 15\log_2(st+\sigma) {n-t\choose k-t}$ sets in the remainder.

Next, recall that we worked in the assumption \eqref{eqsmallk}. Let us deal with the situation when the assumption \eqref{eqsmallk} fails. Also, recall that $k$ is not too small: $k>t+(t/s)^{1/2}+2\sigma+\log_2(16 s^4t^2)$, otherwise we are done in the very beginning. It means that $(t/s)^{1/2}+2\sigma+\log_2(16 s^4t^2)\le k-t\le t^{1/2}$, and the family $\ff$ in a sense is a good low-uniformity approximation for itself. We shall put $ t_1^{(i)}= \lceil 2\log_2 s\rceil$ for steps $i=1,\ldots, 4\log_2 (st)$. Note that, given our lower bound on $k-t$, we have $k-t^{(i)}_1\le 2(k-t)$, and thus we can do the analysis as above. Most importantly, the RHS of \eqref{eqchoicet} in these conditions is again at least \eqref{eqchoicet2} and thus \eqref{eqchoicet} is easily seen to be satisfied. The rest of the analysis is identical.\\

Let us now prove the `more generally' part of the statement of the theorem. Assume that $n\ge Ct^{\alpha}s^{\beta}(k-t)\log_2 \frac{n-t}{k-t}$ with some $\alpha,\beta\ge 1/2$. Then the bound \eqref{boundq} improves to
\begin{multline}\label{boundq2}
  q^{(i)}\le   t_1^{(i)}+2t^{(1-\alpha)/3}s^{-\beta/3}\Big(\frac{q^{(i-1)}-t_1^{(i)}} {100\log_2\frac{n-t}{k-t}}\Big)^{2/3} \\ +2\sigma+\log_2(8s^4t^2)+2(t-t_1^{(i)})\log_2 \frac{n-t}{k-t}
\end{multline}
Thus, we may perform exactly the same analysis as above (only using the $\alpha=\beta=1/2$ case), arriving to $q^{(i_1)}\le t+(t/s)^{1/2}+3\sigma+\log_2(8 s^4t^2)$ and $t_1^{(i_1)} = t$. Then  we run the same procedure with $t^{(i)}=t$ for at most $2\log_2 (st)$ extra steps up until step $i_2$, and for which we use the bound \eqref{boundq2} instead of \eqref{boundq}. It is not difficult to check as before that if for some $i_1\le i<i_2$    $q^{(i)}> t+t^{(1-\alpha)}s^{-\beta}+8\sigma+8\log_2(8 s^4t^2),$ then $q^{(i_2+1)}-t\le \frac 12(q^{(i_2)}-t)$. This is impossible for more than $\log_2 (st)$ steps, which implies that $q^{(i_2)}\le t+t^{(1-\alpha)}s^{-\beta}+8\sigma+8\log_2(8 s^4t^2).$ Applying it for one more step, one sees that $8\sigma+8\log_2(8 s^4t^2)$ gets replaced by an expression that is at most $3\sigma+\log_2(16 s^4t^2)$. The bound on the size of the remainder remains valid since we had a slack in our bound on the number of steps.

This completes the proof of the theorem.

\section{Fine-grained structure in spread approximations} \label{sec4}

In this section, we study the structure in spread approximations of the nearly-extremal Hajnal--Rothschild families.
The following theorems give  a fine-grained (99\%) stability result concerning nearly-extremal
families $\ff$ with $\nu(\ff,t)\le t$. It is the second key step in the proof of our main theorem.
 Their proof occupies the rest of this section. Recall that $\aaa = {[n]\choose k}$.

First we state the harder, non-trivial, case of the theorem.

\begin{thm}\label{thmfinegrained}
   Let $n,k,s,t,\ell$ be positive integers, where $n,k\ge 2$, and
   $t\ge 400\ell^3s$ and $\ell \ge 1$. Suppose that $n> k+C(st)^{1/2}(k-t)$ and $n\ge Cs\ell^4(k-t)$.
   Let $\m S$ be a family of $\le (t+\ell)$-element subsets of $[n]$ that satisfies $\nu(\m S,t)\le s$.
    Let $\ff[\m S]\subset \m A[\m S]$ satisfy $|\ff[\m S]|>  h(n,k,s,t)-
    (\frac 12+\frac 1{10s})h(n,k,1,t)$. Then there exist sets $X_1,\ldots, X_s$ of
    size $t+2x_1,\ldots, t+2x_s,$ with $0\le x_i\le \ell/2$ for all $i\in [s]$, such that,
    putting $\m C_i:={X_i\choose \ell+x_i}$, we have $\m S = \m S[\cup_{i\in[s]} \m C_i]$.
\end{thm}
In other words, the family $\m S$ only contains sets that contain something from one of the {\it cliques} $\m C_i$. The inequality  $t\ge 400\ell^3s$ is in our case necessary for getting into the non-trivial regime. In the case when this inequality is not valid, we shall prove the following theorem.

\begin{thm}\label{thmfinegrained2}
   Let $n,k,s,t,\ell$ be positive integers, where $n,k\ge 2$, and $\ell = 400\log_2 (st)$. Suppose that $ n-t\ge C (t+\ell)\ell(k-t)$, $n\ge k+C(st)^{1/2}(k-t)$ and $n>Cs\ell^4(k-t)$. Let $\m S$ be a family of $\le (t+\ell)$-element sets that satisfies $\nu(\m S,t)\le s$. Let $\ff[\m S]\subset \m A[\m S]$ satisfy $|\ff[\m S]|>  h(n,k,s,t)-(\frac 12+\frac 1{10s})h(n,k,1,t)$. Then there exist sets $T_1,\ldots, T_s$ of size $t$, such that $\m S = \m S[\{T_1,\ldots, T_s\}]$.
\end{thm}

\subsection{Bounding the sizes of the examples}
Let us denote
$$\mathcal E_i:=\bigsqcup_{j\in[s]}{X_j\choose t+i},$$
where $X_1,\ldots, X_s$ are some pairwise disjoint sets of size $t+2i$. We need the following technical lemma. 

\begin{lem}\label{lemboundaei}
  If $n,k,s,t,\ell$ satisfy the assumptions of Theorem~\ref{thmfinegrained} and $i\le \ell$, 
   then we have
\begin{equation}\label{eqboundaei}0.98{t+2i\choose i}{n-t-i\choose k-t-i}\le \frac{|\m A(\m E_i)|}s\le |\m D_{i}|
\le  {t+2i\choose i}{n-t-i\choose k-t-i}.\end{equation}
\end{lem}
\begin{proof}
  It is not difficult to see that using inclusion-exclusion,
  \begin{equation*}|\m A(\m E_i)| = \sum_{j=1}^s (-1)^{j+1}{s\choose j}{n-j(t+2i)\choose k-j(t+i)}.\end{equation*}
  Bonferroni's inequality implies
  \begin{equation}\label{eqsizee}0\ge |\m A(\m E_i)|-s{t+2i\choose i}{n-t-2i\choose k-t-i}\ge -{s\choose 2}{n-2(t+2i)\choose k-j(t+i)}.\end{equation}
We have
  $$\frac{{n-2(t+2i)\choose k-2(t+i)}}{{n-(t+2i)\choose k-(t+i)}}\le \Big(\frac{k-(t+i)}{n-(t+2i)}\Big)^{t+i}\le \Big(\frac{k-t}{n-t}\Big)^{t+i}\le \frac{0.01}s$$
  by $n>Cs(k-t)$. Thus, the right hand size of \eqref{eqsizee} is at least $-0.01s{t+2i\choose i}{n-t-2i\choose k-t-i}$. We also have
  $$\frac{{n-t-i\choose k-t-i}}{{n-t-2i\choose k-t-i}}\le \Big(\frac{n-k}{n-k-i}\Big)^{k-t}\le e^{\frac{\ell(k-t)}{n-k-\ell}}<1.01,$$
  where the last inequality holds since $n\ge k+1000\ell  (k-t)$.

Using the last displayed inequality, we have 
{\small $$0.99 {t+2i\choose t+i}{n-t-i\choose k-t-i}\le {t+2i\choose t+i}{n-t-2i\choose k-t-i}\le|\m D_i|\le {t+2i\choose t+i}{n-t-i\choose k-t-i}$$}
and
 $$\frac{|\m A(\m E_i)|}s\ge 0.99{t+2i\choose t+i}{n-t-2i\choose k-t-i}\ge 0.98 {t+2i\choose t+i}{n-t-i\choose k-t-i}$$

\end{proof}

\subsection{Structure in the spread approximation}
Consider a family $\s\subset {[n]\le t+\ell}$ with $\nu(\m S,t)\le s$ and such that it is an {\it antichain}, i.e., $S_1\not\subset S_2$ for any $S_1\ne S_2\in \m S$. Let us repeatedly apply the following procedure to $\m S$: if there is a set $T$ and  a subfamily $\m T\subset \m S$, such that $\m T(T)$ is $\alpha$-spread for $\alpha>s(\ell+1)$, then we replace in $\m S$ the subfamily of sets containing $T$ with $T$: $\m S:=\m S\setminus \m S[T] \cup \{T\}$. We stop once it is impossible to replace further. We note that $\ff[\m S]$ with the `new' $\m S$ contains $\ff[\m S]$ with the `original' $\m S$. The following claim allows us to `forget' about the initial $\m S$ and work with new $\m S$ (in what follows, $\m S$ stands for the `new' $\m S$).

\begin{cla}
  The resulting family $\m S$ satisfies $\nu(\m S,t)\le s$. (In particular, all sets in $\m S$ are of size $\ge t$.)
\end{cla}
\begin{proof}
  By induction, it is sufficient to show this for one step: if  $\nu(\m S,t)\le s$ and there is a set $T$ and a family   $\m T\subset \m S$, such that $\m T(T)$ is $\alpha$-spread for $\alpha>s(\ell+1)$, then $\nu (\m S\setminus \m S[T] \cup \{T\}, t)\le s$. Also note that all sets in $\m S$ have cardinality $\le t+\ell$.

  Arguing indirectly, assume that there are sets $A_1,\ldots, A_s$, such that $T,$ $A_1,\ldots, A_s$ have pairwise intersections strictly smaller than $t$. For each $j\in [s]$, choose $A'_j\subset A_j\setminus T$, such that $|A'_j| = |A_j|-t+1 \le \ell+1$.  Then any $G\in \m T':=\m T[T, T\cup A_1'\cup\ldots\cup A_s']$ intersects each $A_i$, $i\in [s]$, in at most $t-1$ elements. Indeed, $A_i\cap G\subset A_i\setminus A_i'$, and the latter set has size $t-1$. Provided that $\m T'$ is non-empty and we can choose such a $G$, we get a contradiction with the fact that  $\nu(\m S,t)\le s$. Indeed $G,A_1,\ldots, A_s$ belong to $\m S$ and form a $t$-matching of size $s+1$.  However, we have $|A_1'\cup\ldots\cup A_s'| \le s(\ell+1)<\alpha$, and thus, by the $\alpha$-spreadness of $\m T(X)$, \begin{align*}|\T'|\ge&\ |\m T(T)|-\sum_{x\in \cup_{i=1}^s A'_i} |\m T(T\cup \{x\})|\\
\ge&\ |\m T(X)|- \frac{s(m+1)}\alpha |\m T(T)|>0.\end{align*}

\end{proof}

From now on, we assume that $\m S\subset {[n]\choose \le t+\ell}$ has the property that it has no such $>s(\ell+1)$-spread subfamily. Let us denote $\m W_m:=\m S\cap {[n]\choose t+m}$. Using Observation~\ref{obs13}, we see that, for any fixed set $T$, we have
$$|\m W_m(T)|\le (s(\ell+1))^{m-|T|}.$$
The following analysis is needed for $|\m W_m(T)|$.

Let $\m G\subset{[n]\choose t+m}$ be a family with $\nu(\m G,t)\le s$ and no $\alpha$-spread subfamily with $\alpha>s(m+1)$. We are going to bound the size of $\m G$. Take the largest $t$-matching $\m B:=\{B_1,\ldots, B_{s'}\}\subset \m G$, where $s'\le s$. Next, remove the family $\m G'$ of all sets that intersect at least $2$ sets from $\m B$ in $\ge t$ elements. Put $\m G_1:= \m G\setminus (\m G'\cup \m B)$. Note the following key property of $\m G_1$. For each set $A\in \m G_1$, there is exactly one set $B_i\in\m B$ such that $|A\cap B_i|\ge t$. This implies that $\m G_1 = \m P_1\sqcup \ldots \sqcup \m P_{s'}$, such that $\m P_i$ is the family of all sets from $\m G_1$ that intersect $B_i$ in at least $t$ elements.

\begin{obs} For each $j\in[s']$ the family $\m P_j$ is $t$-intersecting.
\end{obs}
\begin{proof}
  Assume that, say, there are $A, A'\in \m P_1$ such that $|A\cap A'|<t$. Then $A,A',B_2,\ldots, B_{s'}$ have pairwise intersections $<t$, which contradicts the maximality of $\m B$.
\end{proof}

Thus, we get a nice structure in the family $\m G$: it is a union of $s'\le s$ families $\m P_j$ that are each $t$-intersecting, plus the family  $\m G'$, which is hopefully small. Let us bound its size. The family $\m G'$ is contained in the union of the following families, where $A,B\in \m B$:
$$\m G_{A,B}:=\{G\in {[n]\choose k}: |G\cap A|, |G\cap B|\ge t\}.$$
There are ${s'\choose 2}$ families of this form, and, given that $|A\cap B| = t- \beta$, we can bound its size as follows:
$$|\m G_{A,B}|\le \sum_{i=\beta}^m{t-\beta\choose t-i}{m+\beta\choose i}^2 (s(\ell+1))^{m-i},$$
where $\beta \ge 1$.
Let us denote
\begin{equation}\label{eqfi} f_\beta(i):= {t-\beta\choose t-i}{m+\beta\choose i}^2 (s(\ell+1))^{m-i}\end{equation}
and compare $|\mathcal A[\m G_{A,B}]|$ with  $|\m D_{i-\beta}|$ (cf. \eqref{eqak}). We have
$$|\mathcal A[\m G_{A,B}]|\le \sum_{i=1}^m f_\beta(i) {n-t-m\choose k-t-m}$$
and, using Lemma~\ref{lemboundaei},
$$|\mathcal D_{i-\beta}|\ge 0.98{t+2(i-\beta)\choose t+(i-\beta)}{n-t-(i-\beta)\choose k-t-(i-\beta)}.$$
Using that ${t+2(i-\beta)\choose t+(i-\beta)}\ge {t-\beta\choose t-i}$ and ${m+\beta\choose i} = {m+\beta\choose m+\beta-i}\le (m+\beta)^{m+\beta-i}$, we have
\begin{align*}\frac{f_\beta(i) {n-t-m\choose k-t-m}}{|\mathcal D_{i-1}|}&\le \frac{{m+\beta\choose i}^2(s(\ell+1))^{m-i} (k-t)^{m-i+\beta}}{(n-t)^{m-i+\beta}}\\ &\le \frac 1{(s(\ell+1))^\beta}\Big(\frac{e^2 s(\ell+\beta)^3(k-t)}{n-t}\Big)^{m-i+\beta}\\
&\le \frac 1{2(s(\ell+1))}\gamma^{-(m-i+1)},
\end{align*}
where $\gamma>1$, provided \begin{equation}\label{eq432}
n-t>C \gamma s \ell^3(k-t).\end{equation}
Summing over $i$, we get that
$$\frac{|\mathcal A[\m G_{A,B}]|}{\max_i |\m D_i|}\le \sum_{i=1}^m\frac{f_\beta(i) {n-t-m\choose k-t-m}}{|\mathcal D_{i-1}|} \le \frac 1{2s} \gamma^{-1}.$$

Therefore, using Lemma~\ref{lemboundaei}, we have
\begin{equation}\label{eqf1}\frac{|\m A(\m G')|}{\max_i|\m A(\m E_{i})|}<\gamma^{-1}.\end{equation}

We also note that, using virtually identical calculations and for $n-t>C \gamma s \ell^3(k-t)$, we can get for each $i<m$

\begin{equation}\label{eqf0}\frac{s f_0(i){n-t-m\choose k-t-m}}{\max_i|\m A(\m E_{i})|}\le \gamma^{-(m-i)}.
\end{equation}

Our next goal is to relate this to the size of $\W_m$. First, clearly, we may substitute $\m G:=\W_m$ , and get all the conclusions, including the bounds on $|\m G'|$ and the fact that $\m G\setminus \m G'$ is $t$-intersecting.

Let us bound the size of each $\m P_j$. Take two sets $A,B\in \m P_j$ with the smallest intersection, say, $|A\cap B|=t'$. Next, we fix $I\subset A\cap B$, $|I|=t$ and put $D_1 = A\setminus B$, $D_2 = B\setminus A$. Note that $|D_1|,|D_2|\le m$. We group other sets $C$ from $\m P_j$ depending on their intersection with $I$. Note that if $|C\cap I| = t-m$, then $|C\cap (A\cap B)|=|A\cap B|-m$. Since $|C\cap A|, |C\cap B|\ge |A\cap B|$ by the choice of $A,B$, we must have $|C\cap D_1|, |C\cap D_2|\ge m$. Since $|C|=t+m$, no elements are left and thus $C\subset A\cup B$. In the first inequality below we use that $|\m P_j(X)|\le (s(\ell+1))^{t+m-|X|}$ for any set $X$ by our assumption on $\m S$ that it has no $\alpha$-spread subfamilies and Observation~\ref{obs13}. We get that
\begin{align}\notag|\m P_j|&\le \Big|\m P_j\cap {A\cup B\choose t+m}\Big|+\sum_{i=0}^{m-1} {t\choose t-i}{|D_1|\choose i}{|D_2|\choose i}(s(\ell+1))^{m-i}\\
\notag&\le  \Big|\m P_j\cap {A\cup B\choose t+m}\Big|+\sum_{i=0}^{m-1} {t\choose t-i}{m\choose i}^2(s(\ell+1))^{m-i}\\
\label{eqpi} &=  \Big|\m P_j\cap {A\cup B\choose t+m}\Big|+\sum_{i=0}^{m-1} f_0(i).\end{align}

\subsection{Proof of Theorem~\ref{thmfinegrained}} We prove the statement by induction on $s$, in which we use both the statement of Theorem~\ref{thmfinegrained} for $s-1$ and also\footnote{The latter is for convenience only. It is possible to avoid, but then some extra calculations will be required.}  the statement of Theorem~\ref{thmmain} for $s-1$. The proof of Theorem~\ref{thmmain} for $s-1$ only requires Theorem~\ref{thmfinegrained} for $s-1$ and, moreover, the assumed inequalities on the parameters, if valid for $s$, stay valid for $s-1$. Thus, this creates no problem.

Let $m\in [\ell]$ be the largest $i$ such that $|\ff[\m W_i]|\ge \frac 1{20\ell}|\ff[\m S]|$. Since $\m S\subset \sqcup_{i=0}^\ell \m W_i$, such $m$ clearly exists. 

The analysis from the previous section makes it possible for us to compare $\max_i|\m A(\m E_i)|$ with $|\ff(\m W_m)|$. We use the decomposition of $\m W_m=\m G' \sqcup \sqcup_{j=1}^{s'}\m P_j$ as in the previous section and the fact that, for any family $\m Y\subset {[n]\choose t+m}$, we have $|\ff(\m Y)|\le |\m A(\m Y)|\le |\m Y|{n-t-m\choose k-t-m}$. We shall be treating a part of the families $\m P_j$ separately. Let us put $\m P_j':=\m P_j\cap {A\cup B\choose t+m}$, where $A$ and $B$ are as in the previous section (cf. \eqref{eqpi}). We get the following chain of inequalities.
\begin{align}\notag|\ff(\m W_m)|\le& \sum_{j=1}^{s'}|\ff(\m P_j')|+\Big|\m G' \sqcup \bigsqcup_{j=1}^{s'}\m P_j\setminus \m P_j'\Big|{n-t-m\choose k-t-m} \\
\notag \overset{\eqref{eqf1},\eqref{eqpi}}{\le}& \sum_{j=1}^{s'}|\ff(\m P_j')|+\gamma^{-1}\max_i|\m A(\m E_{i})|+s\sum_{i=0}^{m-1} f_0(i){n-t-m\choose k-t-m} \\
\notag\overset{\eqref{eqf0}}{\le}& \sum_{j=1}^{s'}|\ff(\m P_j')|+\gamma^{-1}\max_i|\m A(\m E_{i})|+\max_i|\m A(\m E_{i})|\sum_{i=0}^{m-1}\gamma^{i-m}\\
\label{eq117}\le& \sum_{j=1}^{s'}|\ff(\m P_j')|+\frac 1{100\ell}\max_i|\m A(\m E_{i})|,
 \end{align}
Where in the last inequality we chose $\gamma = 300\ell$, and we require $n>Cs \ell^4 (k-t)$ (cf. \eqref{eq432}).
Given that $|\ff[\m W_m]|\ge \frac 1{20\ell}|\ff[\m S]|$ and $$|\ff[\m S]|\ge  h(n,k,s,t)-(\frac 12+\frac 1{10s})h(n,k,1,t)\ge \frac 25h(n,k,s,t)\ge  \frac 25\max_i|\m A(\m E_{i})|,$$ we get that $|\ff(\m W_m)|\ge \frac 1{50\ell}\max_i|\m A(\m E_{i})|$.  Combining this with \eqref{eq117} and Lemma~\ref{lemboundaei}, we conclude that
\begin{multline*}s\max_j|\ff(\m P_j')| \ge \frac 1{100\ell}\max_i|\m A(\m E_{i})|\\
\ge \frac 1{100\ell}|\m A(\m E_{m})|\ge \frac s{125\ell}{t+2m\choose m}{n-t-m\choose k-t-m}.\end{multline*}
Assuming that $\max_j|\ff(\m P_j')| = |\ff(\m P_1')|$ and that $\m P_1'\subset {A\cup B\choose t+m}$, we get that
\begin{equation}\label{eqab}\big|\m W_m\cap {A\cup B\choose t+m}\big|\ge \frac 1{125 \ell}{t+2m\choose m}.\end{equation}
Given our bound on $\ell$, we in particular see that $|A\cup B| = t+2m$. It was a theoretical possibility that this union is smaller, but then ${A\cup B\choose t+m}\le {t+2m-1\choose t+m}=\frac{m}{t+2m}{t+2m\choose t+m},$ which contradicts \eqref{eqab} since $t> 200\ell m$. Put $\m C_1:= {A\cup B\choose t+m}$ and note that $\m C_1$ satisfies $\nu(\m C_1,t)=1$. Summarizing, we see that $\W_m$ is rather dense in $\m C_1$. We capitalize on this in the following lemma.

\begin{lem}\label{leminduction}
  Put $\m S^{(1)}:=\m S\setminus \m S[\m C_1]$. We have $\nu(\m S^{(1)},t)\le s-1$.
\end{lem}
\begin{proof}
  Assume that there are sets $A_1,\ldots, A_s\in \m S^{(1)}$ such that $|A_i\cap A_j|<t$ and,
  moreover, $|A_i\cap (A\cup B)|=m+t-x_i,$ $x_i>0$. To arrive at a contradiction, we just need to find
   a set $X\in\m W_m\cap \m C_1$ that intersects $A_i$ in fewer than $t$ elements.

Recall that sets in $\m S$ have size at most $t+\ell$. W.l.o.g., assume that $x_1,\ldots, x_{s'}\le \ell$ and $x_{s'+1},\ldots, x_{s}> \ell$.  Note that, no matter what set $C\in \m S[\m C_1]$ we take, we have $|C\cap A_i|<t$ for each $i>s'$. Indeed, $|C\cap A_i|\le |(A\cup B)\cap A_i|+|C\setminus (A\cup B)|<t+m-\ell +(\ell-m)=t$.   Thus, we only need to concentrate on sets $A_i, i\in [s']$.

  For each $i\in [s']$ we have
   $$\frac{{t+2m-x_i\choose t+m-x_i}}{{t+2m\choose t+m}}\ge\frac{{t+2m-\ell\choose t+m-\ell}}{{t+2m\choose t+m}}\ge \Big(\frac{t+m-\ell}{t+2m-\ell}\Big)^{\ell}\ge 1-\frac{m\ell}{t-\ell}\ge 1-\frac 1{300\ell s},$$
  since $t\ge 400 \ell^3 s$ by assumption.

  For each $A_i, i\in [s'],$ consider the set $Z_i:=(A\cup B)\setminus A_i$ and note that, by the above calculation, at least a $(1-\frac 1{300\ell s})$-proportion of all sets in ${A\cup B\choose t+m}$ contain $Z_i$. By the union bound, at least $1-\frac 1{300\ell}$ fraction of sets in  ${A\cup B\choose t+m}$ contain $Z_i$ for each $i\in [s']$. Note that any such set intersects $A_i$, $i\in[s']$, in $t-x_i$ elements.

  Combining the above with \eqref{eqab}, we see that there is a set $X\in \m W_m$ that intersects $A_i$, $i\in [s']$, in $t-x_i$ elements. We have $X\in \m S$ and we are done: together with $A_1,\ldots, A_s$ it forms a $t$-matching of size $s+1$.
    \end{proof}

Now, we are almost in position to apply induction to the family $\m S^{(1)}$. We have $\nu(\m S^{(1)},t)\le s-1$ by Lemma~\ref{leminduction}. The only thing that we need to check is that  $|\ff[\m S^{(1)}]|\ge h(n,k,s-1,t)-\big(\frac 12 +\frac 1{10(s-1)}\big)h(n,k,1,t)$. This is a bit technical, albeit simple, and is done in  Lemma~\ref{lemboundh2} part 2 below. This concludes the proof of the theorem.

The following is  an immediate consequence of Theorem~\ref{thmmain} for $s-1$ (which, as we mentioned in the beginning of this proof, we use in the inductive step for convenience).
\begin{obs}\label{lemboundh}
  In our assumptions on $n,k,t,s$, we have $h(n,k,s-1,t)\le (s-1) h(n,k,1,t)$.
\end{obs}


The following lemma provides us with lower bounds for $h(n,k,s,t)$ as well.

\begin{lem}\label{lemboundh2}
Let $t\ge 8$ and $n\ge k+C(st)^{1/2}(k-t)$.

1.  We have $$h(n,k,s,t)\ge h(n,k,s-1,t)+\Big(1-\frac{1}{10s^3}\Big)h(n,k,1,t).$$ 

2. If  the families $\m U_i$ satisfy $\nu(\m U_1,t)= s-1$ and $\nu(\m U_2,t)= 1$ and, moreover, $|\m U_1\cup \m U_2|\ge h(n,k,s,t)-(\frac 12+\frac 1{10s}) h(n,k,1,t)$ then $|\m U_1|\ge h(n,k,s-1,t)-(\frac 12+\frac 1{10(s-1)})h(n,k,1,t)$.
\end{lem}
\begin{proof}
1. Recall that a family $\m G$ is $t$-intersecting iff $\nu(\m G,t)=1$.
It is a result of the first author \cite{F2020} that $h(n,k,1,t)\le {n-1\choose k-t}$, and thus, using Observation~\ref{lemboundh}, we have $h(n,k,s-1,t)\le (s-1){n-1\choose k-t}$. Consider a family $\m G$ with $\nu(\m G,t)=s-1$ and $|\m G| = h(n,k,s-1,t)$. Consider also the largest $t$-intersecting family $\m G'\subset {[n]\choose k}$. Next, take a random permutation $\sigma$. Then
$$\E[|\sigma(\m G')\cap \m G|] = \frac{|\m G||\m G'|}{{n\choose k}}\le \frac {(s-1){n-1\choose k-t}}{{n\choose k}}|\m G'|\le \frac {s(k-t+7)^{8}}{(n-k)^{8}}|\m G'| \le \frac{1}{10s^3}|\m G'|,$$
where the last inequality is due to our assumption on $n$. Take a permutation $\sigma$ for which the size of the intersection is at most its expectation. Note that $\nu(\sigma(\m G')\cup \m G,t)\le s$ and, therefore,
$$h(n,k,s,t)\ge |\sigma(\m G')\cup \m G|\ge h(n,k,s-1,t)+\Big(1-\frac{1}{10s^3}\Big)h(n,k,1,t).$$


2. In order to obtain the bound, we simply substitute the bound from part 1 of the lemma in the following chain of inequalities:
\begin{align*} |\m U_1|\ge& |\m U_1\cup \m U_2|-|\m U_2|\ge |\m U_1\cup \m U_2|-h(n,k,1,t)\\
\ge& h(n,k,s,t)-\Big(\frac 32+\frac 1{10s}\Big)h(n,k,1,t)\\ \ge& h(n,k,s-1,t)-\Big(\frac 12+\frac 1{10s}+\frac{1}{10s^3}\Big)h(n,k,1,t)\\
\ge& h(n,k,s,t)-\Big(\frac 12+\frac 1{10(s-1)}\Big)h(n,k,1,t).\end{align*}
 \end{proof}

\subsection{Proof of Theorem~\ref{thmfinegrained2}}
We use a similar, albeit simpler, proof strategy. We have $n\ge C\ell^4 s (k-t)$ and may use the same analysis as in the proof of Theorem~\ref{thmfinegrained} up to \eqref{eq117}. Next, we want to show that $m=0$ in \eqref{eq117}. Arguing indirectly, if $m\ge 1$, we get that
$$\frac{s\max_j|\ff(\m P_j')|}{|\m A(\m E_0)|}\le 2\frac{{t+2m\choose m}{n-t-m\choose k-t-m}}{{n-t\choose k-t}}\le 2\Big(\frac{(t+2m)(k-t)}{n-t}\Big)^m\le \frac 1{1000\ell},$$
since $n-t\ge C(t+m)\ell(k-t)$ in our assumptions. This, together with \eqref{eq117}  contradicts the inequality $|\ff(\m W_m)|\ge \frac 1{50\ell}\max_i|\m A(\m E_{i})|$.

Concluding, we get that $m=0$ and, using the same notation as in the proof of Theorem~\ref{thmfinegrained}, we see that $\m S^{(1)} = \{T_1\}$ for some $t$-element set $T$. In this case, we immediately get that for $\m S^{(1)}:=\m S\setminus \m S[T_1]=\m S\setminus \{T\}$ we have
$\nu(\m S^{(1)})\le s-1$ and apply induction.

\section{Proof of the main theorem}\label{secmain}
The first step of the proof is a combination of the application of Theorem~\ref{thmapprox1} and Theorems~\ref{thmfinegrained},~\ref{thmfinegrained2}. Slightly abusing notation, we use $C$ below for different large enough constants that we can choose (culminating with the choice of the constant $C$ in the condition on $n$ in Theorem~\ref{thmmain}).

Assume first that $t = s^{y}$, $y>1$ and that $t\ge C s\log^3_2(st)$ for a sufficiently large $C$. Apply Theorem~\ref{thmapprox1} with $\beta = 0.5$ and $\alpha = \frac 45 -\frac 3{10y}$, and $\sigma = 100(t/s)^{1/2}$. The parameters are chosen so that $t^\alpha s^{\beta} = t^{4/5}s^{1/5}$. Note that the constraints on $n$ from Theorem~\ref{thmapprox1} are satisfied (the bound $n\ge k+C(k-t)s(\sigma+\log_2 n) = k+C(k-t)(st)^{1/2}+C(k-t)s\log_2 n $ is implied by our condition on $n$). We get a family $\m S$ with $\nu(\m S,t)\le s$ that approximates $\m F$ and that has sets of size at most $t+\ell$, where $\ell \le  t^{1-\alpha} s^{-\beta}+400\log(st) = (t/s)^{1/5}+400\log(st).$ We also get a remainder $\m R$ of size at most $2^{-90(st)^{1/2}}{n-t\choose k-t}$.

Next, we apply Theorem~\ref{thmfinegrained} to $\m S$. Note that the condition $n\ge Cs\ell^4(k-t)$ is implied by $n\ge Cs(t/s)^{4/5}(k-t)=Ct^{4/5}s^{1/5}(k-t)$ and $n\ge Cs(k-t)\log^4 n$, which we assumed. The condition $t\ge 400\ell^3s$ holds because we assumed that $t\ge C s\log^3(st)$. Since $\ff$ is extremal and $|\m R| = |\ff\setminus \ff[\m S]|\le 0.01 h(n,k,1,t)$, the condition on $|\ff[\m S]|$ from Theorem~\ref{thmfinegrained} is satisfied. We conclude that there are sets $X_1,\ldots, X_s$ of size $t+2x_1,\ldots, t+2x_s,$ with $0\le x_i\le \ell$ for all $i\in [s]$, such that,
    putting $\m C_i:={X_i\choose \ell+x_i}$, we have $\m S = \m S[\cup_{i\in[s]} \m C_i]$.

That is, in order to complete the proof of the theorem in this case, we are only left to show that $\m R=\emptyset$. But let us first deal with the remaining values of the parameters.

Assume that $t\le C s\log^3_2(st)$.  Apply Theorem~\ref{thmapprox1} with $\beta = 0.5$ and $\alpha = 0.5$ if $t\le s$ and $\alpha=1-0.5\log_t s$ otherwise, and $\sigma = 100\log_2(st)$. This way, we have $t^{1-\alpha}s^{-\beta} = 1$. Note that $t^{\alpha}s^{\beta}\le \max\{s,t^{4/5}s^{1/5}\}$, 
and thus the constraint on $n$ from Theorem~\ref{thmmain} implies $n\ge t+C t^{\alpha}s^{\beta}(k-t)\log_2\frac{n-t}{k-t}$. The other constraint is also satisfied.
We get a family $\m S$ with $\nu(\m S,t)\le s$ that approximates $\m F$ and that has sets of size at most $t+\ell$, where $\ell \le  400\log_2(st)$. We also get a remainder $\m R$ of size at most $(st)^{-90}{n-t\choose k-t}$.

Next, we apply Theorem~\ref{thmfinegrained2} to $\m S$. Note that the condition $n\ge Cs\ell^4(k-t)$ is implied by $n\ge Cs(k-t)\log^4n$, which we assumed. The condition $n-t\ge C(t+\ell)\ell(k-t)$ holds because $(t+\ell)\ell(k-t)\le Cs \ell^4(k-t)$. The condition $n\ge k+C(st)^{1/2}(k-t)$ is also valid. Again, since $\ff$ is extremal and $|\m R| = |\ff\setminus \ff[\m S]|\le 0.01 h(n,k,1,t)$, the condition on $|\ff[\m S]|$ from Theorem~\ref{thmfinegrained2} is also satisfied. We conclude that there exist sets $T_1,\ldots, T_s$ of size $t$, such that $\m S = \m S[\{T_1,\ldots, T_s\}]$.

In what follows, we shall argue that the remainder $\m R$ must be empty in order for $\ff$ to be extremal. Let us again first deal with the harder case of $t\ge Cs\log^3(st)$.

Find an inclusion-maximal set $Y$ such that $\m R(Y)$ is $r$-spread for $r = \frac {n-t}{2(k-t)}$.  Using the by now standard calculations as in the proof of Theorem~\ref{thmregularity}, one can see that $|\m R|\le r^{|Y|}{n-|Y|\choose k-|Y|}$. We shall use this bound for $|Y|> 2t\log_2n$. In this regime it implies
\begin{align}\notag|\m R|&\le r^{|Y|}{n-|Y|\choose k-|Y|} \le \Big(\frac {n-t}{2(k-t)}\Big)^{|Y|}\Big(\frac{k-t}{n-t}\Big)^{|Y|-t}{n-t\choose k-t}\\
\label{eq235}&=2^{-|Y|)}\Big(\frac{n-t}{k-t}\Big)^{t}{n-t\choose k-t}\ge 2^{-(|Y|-2t\log_2 n)}n^{-t}{n-t\choose k-t}.\end{align}

 As long as $|Y|\le 2t\log_2 n$, it is sufficient for us to use the bound $|\m R|\le 2^{-90(st)^{1/2}}{n-t\choose k-t}$, guaranteed by the application of Thm~\ref{thmapprox1} above.

We may of course w.l.o.g. assume that $\m R\cap  \bigcup_{i=1}^s\m A[\m C_i]=\emptyset$, and thus $|Y\cap X_i|\le \ell+x_i-1$ for each $i\in [s]$. Fix some set $Y_i, Y\cap X_i\subset Y_i\subset X_i$, so that $|Y_i| = \ell+x_i-1$. Denote $Z_i:=X_i\setminus Y_i$. Consider the following families
$$\m P_i:=\Big\{F\in \aaa[\m C_i]: Z_i\subset F, |F\cap X_i|= t+x_i, F\cap Y\subset Y_i\Big\}.$$
Note that, first, for any $P\in \m P_i$ we have $|P\cap Y|\le t-1$ and, second,
\begin{align}\notag|\m P_i|&\ge {t+x_i-1\choose t-1}{n-t-2x_i-|Y|\choose k-t-x_i}\\
\notag&\ge \Big(\frac{t}{t+x_i}\Big)^{x_i+1}{t+2x_i\choose t+x_i}\Big(\frac{n-k-x_i-|Y|}{n-k-x_i}\Big)^{k-t-x_i}{n-t-2x_i\choose k-t-x_i}\\
\label{eq227}&\ge e^{-\frac{x_i(x_i+1)}t}{t+2x_i\choose t+x_i}e^{-\frac{|Y|(k-t)}{n-k-x_i-|Y|}}{n-t-2x_i\choose k-t-x_i}\\
\label{eq228} &\ge \frac 12 e^{-\frac{|Y|}{C(st)^{1/2}\log_2 n}}{t+2x_i\choose t+x_i}{n-t-2x_i\choose k-t-x_i}\\
\label{eq229} &\ge \frac 13 e^{-\frac{|Y|}{C(st)^{1/2}\log_2 n}}{t+2x_i\choose t+x_i}{n-t-x_i\choose k-t-x_i},\\
\label{eq230} &\ge \frac 14 e^{-\frac{|Y|}{C(st)^{1/2}\log_2 n}}{n-t\choose k-t},
\end{align}
where in \eqref{eq228} we used that $t\ge 2\ell^2\ge 4x_i^2$ and our assumption on $n$, and in \eqref{eq229} we used the calculations as in Lemma~\ref{lemboundaei}. In \eqref{eq230} we used that $X_i$ is such that the cardinality of $\aaa[X_i] = \m D_i$ is within $0.99$ multiplicative factor of $\max_i |\m D_i|\ge |\m D_0| = {n-t\choose k-t}$. Otherwise, $|\m F|$ is not maximal because of Lemma~\ref{lemboundh2} part 1 (and the bound on the remainder $|\m R|$).

Note that if $|Y|\le 2t\log_2 n$ the expression in \eqref{eq230} is at least
$$\frac 14 e^{-\frac{2t\log_2 n}{C(st)^{1/2}\log_2 n}}{n-t\choose k-t}\ge \frac 14 e^{-(t/s)^{1/2}}{n-t\choose k-t}.$$
An important thing  to keep in mind is that this is much more than the size of $|\m R|$, which is at most $2^{-90(st)^{1/2}}{n-t\choose k-t}.$

Similarly, if $|Y|\ge 2\log_2 n$ then  the expression in \eqref{eq230} is at least
\begin{multline*}\frac 14 e^{-\frac{2t\log_2 n}{C(st)^{1/2}\log_2 n}}e^{-\frac{|Y|-(2t\log_2 n)}{C(st)^{1/2}\log_2 n}}{n-t\choose k-t}\\ \ge \frac 14 e^{-(t/s)^{1/2}}{n-t\choose k-t} e^{-\frac{|Y|-2t\log_2 n}{100000}}{n-t\choose k-t},\end{multline*}
which is much more than the size of $\m R$, since we have $$|\m R|\le n^{-t}2^{-|Y|-(t+2t\log_2 n)}{n-t\choose k-t}.$$

In either case, we can see that
\begin{equation}\label{eqratiorem} \frac{|\m P_i|}{|\m R|}\ge 2^{80(st)^{1/2}}.\end{equation}

This implies several things. First, consider a set $Q\in {X_i\choose t+x_i}$ and note that $\m P_i = \sqcup \aaa[Q, X_i\cup Y]$, where the union is taken over ${t+x_i-1\choose t-1}$ sets $Q\in {X_i\choose t+x_i}$ that contain $Z_i$. The family $\ff$ is extremal, so, in view of \eqref{eqratiorem}, most of the sets from $\aaa[\m P_i]$ must also lie in $\ff[\m P_i]$. In particular,  for all but at most $0.01$ fraction of sets $Q\in {X_i\choose t+x_i}$, we have  $|\ff(Q, X_i\cup Y)|\ge \frac 12 |\aaa(Q, X_i\cup Y)|$  and $Z_i\subset Q$. The latter is because
 \begin{equation}\label{boundz}\frac{|\m C_i\setminus \m C_i[Z_i]|}{|\m C_i|} = \frac{{t+2x_i\choose x_i}-{t+x_i-1\choose x_i-1}}{{t+2x_i\choose x_i}}\le \frac{(x_i+1){t+2x_i-1\choose x_i}}{{t+2x_i\choose x_i}} \le  \frac{2x_i^2}{t}\le \frac 1{200s},\end{equation}
valid since $x_i\le \ell$ and, by choice of parameters we have $t\ge 400\ell^3 s$ (cf. application of Theorem~\ref{thmfinegrained} in the beginning of the proof).

Let us denote $\m Q_i$ the family of those $Q\in {X_i\choose t+x_i}$, such that $Z_i\subset Q$ and $|\ff(Q, X_i\cup Y)|\ge \frac 12 |\aaa(Q, X_i\cup Y)|$.

Next, because of our bound on the remainder $\m R$, the families $\m C_i$ must be nearly disjoint. Namely, we have shown in Lemma~\ref{lemboundh2} that we must have  $h(n,k,s,t)\ge h(n,k,s-1,t)+(1-\frac 1{10s^3})h(n,k,1,t)$. This and the fact that $\ff$ is extremal implies that for any two $\m C_i, \m C_j$ with $i\ne j$ we must have $|\m A[\m C_i]\cap \m A[\m C_j]|\le \frac 1{5s^3}|\m A[\m C_i]|$.

Using this, we conclude that $\m Q_i':=\m Q_i\setminus\{Q: Q\in \m A[\m C_j]\text{ for some }j\ne i\}$ contains at least a $0.98$-fraction of all sets $Q\in {X_i\choose t+x_i}$. Next, for each $i\in [s]$ we shall select a set $Q_i\in \m Q_i'$ such that $|Q_i\cap Q_j|\le t-1$ for each $i\ne j\in[s]$.
\begin{itemize}
  \item For $i=1$ select any $Q_1\in \m Q_1'$.
  \item Given that we selected $Q_1,\ldots, Q_{i-1}$, for $j\in [i-1]$ denote $V_{ij}=X_i\setminus Q_j$ and note that $|V_{ij}|\ge x_i+1$ because $Q_j\notin \aaa[\m C_i]$. Choose some set $V'_{ij}\subset V_{ij}$ of size exactly $x_i+1$ and note that
      $\frac{|\m C_i\setminus \m C_i[V'_{ij}]|}{|\m C_i|} \le \frac 1{200s}$ via the same calculations as in \eqref{boundz}. From here, using the union bound, we conclude that at least a $0.99$-fraction of all sets in $\m C_i$ contain $V_{ij}'$ for each $j\in [i-1]$. Thus there is such a set $Q_i$, which additionally belongs to $\m Q_i'$. Note that, for any $j\in [i-1]$,  $|Q_i\cap Q_j|\le |Q_i\setminus V'_{ij}| =t-1$. This completes our construction.
\end{itemize}
Note that the sets $Y, Q_1,\ldots, Q_s$ have pairwise intersections $\le t-1$ and, moreover, $\m R(Y)$ and $\m F'_i:=\m F(Q_i, X_i\cup Y)$ are all $r$-spread for $r=\frac {n-2k}{2(k-t)}$. For $\m R(Y)$ this is true by the choice of $Y$, and for each $\m F'_i$ it is true since $|\m F(Q_i, X_i\cup Y)|\ge \frac 12\m A(Q_i, X_i\cup Y)$, and the latter is $\frac{n-2k}{k-t}$-spread.

We can essentially repeat the argument from Lemma~\ref{lemtint} with $t'=t$ in order to get $s+1$ sets $Y''\in \m R\subset \m F$ and $Q_i''\in \ff[Q_i, X_i\cup Y]\subset \ff$ with pairwise intersection strictly smaller than $t$. For $i\ne j\in [s]$ put $a_{ij}:=t-|Q_i\cap Q_j|$ and also $a_{0j}:=t-|Y\cap Q_j|$. For $i\in [s]$ consider
  \begin{align*}\g_i &:=\ff'_i\setminus\Big\{F\in \ff'_i: \ \exists j\in[s]\setminus \{i\}, \text{ such that } |F\cap  (Q_j\setminus Q_i)|\ge \frac{a_{ij}}2\Big\},\\
  \m R'&:= \m R(Y)\setminus\Big\{F\in \m R(Y): \ \exists i\in[s], \text{ such that } |F\cap  Q_i|\ge a_{0i}\Big\}.\end{align*}

   The size of the former family is at most
 \begin{align*}\sum_{j\in [s]\setminus\{i\}}{| Q_j\setminus Q_i|\choose \big\lceil\frac{a_{ij}}2\big\rceil}\max_{X: |X| = \big\lceil\frac{a_{ij}}2\big\rceil, X\cap Q_i = \emptyset}|\ff_i'(X)|\le&\\  \sum_{j\in [s]\setminus\{i\}}{x_j+a_{ij}\choose \big\lceil\frac{a_{ij}}2\big\rceil} r^{-\big\lceil\frac{a_{ij}}2\big\rceil} |\ff'_i|\le&\\
  \sum_{j\in [s]\setminus\{i\}} \Big(\frac{e(\ell+a_{ij})} {\big\lceil\frac{a_{ij}}2\big\rceil} \Big)^{\big\lceil\frac{a_{ij}}2\big\rceil} r^{-\big\lceil\frac{a_{ij}}2\big\rceil}|\ff_i'|\le&\\
  \sum_{j\in [s]\setminus\{i\}} (2e(\ell+1))^{\big\lceil\frac{a_{ij}}2\big\rceil} r^{-\big\lceil\frac{a_{ij}}2\big\rceil}|\ff_i'|
  \le&\\
  (s-1)\frac{2e(\ell+1)}{r} |\ff_i'|\le&\ \frac 12 |\ff_i'|,\end{align*}
  where we pass from the third to the fourth line using $\frac{e(\ell+a_{ij})} {\lceil\frac{a_{ij}}2\rceil}\le \frac{2e(\ell+a_{ij})} {a_{ij}}\le 2e(\ell+1)$ since  $a_{ij}\ge 1$. In the last two displayed inequalities we use $r\ge 100s\ell$.

  This implies that $|\g_i|\ge \frac 12 |\ff'_i|$. Again,
 because of this and the trivial inclusion $\g_i(Z)\subset \ff'_i(Z)$, valid for any $Z$,  we conclude that $\g_i$ is $ \frac r2$-spread for all $i\in [s]$. Using a virtually identical chain of inequalities, we can show that $|\m R'|\ge \frac 12|\m R(Y)|$ and thus that $\m R'$ is $r/2$-spread.

 Now, importantly, if we find pairwise disjoint sets $Y'\in \m R'$, $Q'_i\in \m G_i$, then, putting $Y'':=Y\cup Y'$, $Q_i'':=Q_i\cup Q'_i$, we have $Y''\in \m R, Q_{i}''\in \ff[Q_i, X_i\cup Y]$ and, crucially,
 \begin{align*}|Y''\cap Q_i''| &= |Y'\cap Q_i|+|Y\cap Q_i|+|Y\cap Q'_i| < a_{0i}+(t-a_{0i})+0=t,\\
|Q_i''\cap Q_j''| &= |Q_i'\cap Q_j|+|Q_i\cap Q_j|+|Q_i\cap Q_j'|<\frac {a_{ij}}2+(t-a_{ij})+\frac{a_{ij}}2 = t.
\end{align*}
Thus, we arrive at a contradiction, as long as we can guarantee pairwise disjoint sets $Y', Q'_i$, $i\in[s]$. This is done as in Lemma~\ref{lemtint}. Concluding, we must have $\m R =\emptyset$ in this case.\\

Let us now argue in the case $t\le Cs\log^3_2(st)$. The proof in this case is very similar and is actually easier. We go along the proof and point out the changes made. We define $Y$ in the same way and note that $|Y\cap T_i|<t$ (i.e., $T_i\not\subset Y$). Again, for $|Y|\ge 2t\log_2n$ we use the bound \eqref{eq235}. As long as $|Y|\le 2t\log_2n$, we use $|\m R|\le (st)^{-90}{n-t\choose k-t}$. We define $\m P_i:=\aaa[T_i, T_i\cup Y]$ and replace the calculations \eqref{eq227}--\eqref{eq230} by the following.
$$|\m P_i|={n-t-|Y|\choose k-t}\ge \Big(\frac{n-k-|Y|}{n-k}\Big)^{k-t}{n-t\choose k-t}\ge e^{-\frac{|Y|(k-t)}{n-k-|Y|}}{n-t\choose k-t}.$$
If $|Y|\le 2t\log_2n$, then from the above we get $|\m P_i|\ge \frac 12{n-t\choose k-t}$ since $n>2k+Cs(k-t)\log^4_2n>2k+Ct(k-t)\log_2n$ and (using it in the denominator of the exponent) $|Y|\le k$. If $|Y|>2t\log_2n$, we can similarly get $|\m P_i|\ge 2^{-(|Y|-2t\log_2 n)/(1000)} \frac 12{n-t\choose k-t}$. Combining with \eqref{eq235}, we get the following analogue of \eqref{eqratiorem}.
$$\frac{|\m P_i|}{|\m R|}\ge (st)^{80}.$$
The next step is also significantly easier. Since $\ff$ is extremal, using the above, we may simply say that $\m F'_i:=\ff(T_i, T_i\cup Y)$ have size at least half of the size of $\m P_i$, and thus $\m F'_i$ are $r$-spread with $r = \frac{n-k-t}{2(k-t)}$. What is left is to repeat the argument in the style of Lemma~\ref{lemtint}. Virtually no change in needed in that argument. The theorem is proved.

\end{document}